\documentclass{article}
\usepackage[utf8]{inputenc}
\usepackage{lineno,hyperref}
\modulolinenumbers[5]
\usepackage{amsfonts, amssymb, amsmath}
\usepackage{amsthm}
\usepackage{graphicx}
\usepackage{multirow}
\usepackage{geometry} 
\usepackage{xcolor}
\usepackage[english]{babel}
\usepackage{textcomp}
\usepackage{gensymb}
\usepackage{colortbl}
\usepackage{hyperref}
\usepackage{mathbbol}
\usepackage{subcaption}

\def\etat{\psi}
\def\dd{d}
\def\FR{\mathcal F}
\def\XR{\mathcal X}
\def\HR{\mathcal H}
\def\TR{\mathcal T}
\def\RR{\mathbb R}
\def\NN{\mathbb N}
\def\CC{\mathbb C}
\def\OO{\mathcal O}
\def\Id{I}
\def\cont{u}
\def\vit{G}
\def\cste{C_c}
\def\cs{C_s}
\def\mes{y}
\newcommand{\intg}{\mathfrak{G}}
\newcommand{\xmin}{x_{\min}}
\newcommand{\xmax}{x_{\max}}
\renewcommand{\Im}{\textup{\texttt{Im}}}
\newcommand{\argmin}{\textup{\texttt{argmin}}}

\renewcommand{\leq}{\leqslant}
\renewcommand{\geq}{\geqslant}

\newcommand{\regis}{\textsuperscript{\mbox{\scriptsize{\textregistered}}} }
\newcommand{\regisp}{\textsuperscript{\mbox{\scriptsize{\textregistered}}}}
\newcommand{\fbrm}{FBRM\regis}
\newcommand{\fbrmp}{FBRM\regisp}

\newtheorem{theorem}{Theorem}[section]
\newtheorem{proposition}[theorem]{Proposition}
\newtheorem{remark}{Remark}

\def\startlucas{\begingroup} 
\def\stoplucas{\endgroup}

\usepackage{authblk} 
\usepackage[numbers]{natbib}

\title{New dynamical observer for a batch crystallization process based on solute concentration}

\date{\today}

\author[1]{Lucas Brivadis}
\author[1]{Vincent Andrieu}
\author[1]{Élodie Chabanon}
\author[1]{Émilie Gagnière}
\author[1]{Noureddine Lebaz}
\author[1]{Ulysse Serres}

\affil[1]{Univ Lyon, Universit\'e Claude Bernard Lyon 1, CNRS, LAGEPP UMR 5007, 43 boulevard du 11 novembre 1918, F-69100, Villeurbanne, France}

\begin{document}

\maketitle

\begin{abstract}

In this paper a new observer is introduced to estimate the Crystal Size Distribution (CSD) only from the measurements of the solute concentration, temperature
\startlucas
and a model of the growth rate.
No model of the nucleation rate is needed.
This approach is based on the use of a Kazantzis-Kravaris/Luenberger observer which exponentially estimates functionals of the CSD.
Then, the full state is estimated by means of a Tikhonov regularization procedure.
Numerical simulations are provided.
Our approach relies on an infinite-dimensional observer, contrarily to the usual moment based observers.
\stoplucas

\end{abstract}

{\small
\startlucas
Crystal Size Distribution,
\stoplucas
Tikhonov regularization, Solute concentration
}


\setlength{\baselineskip}{1.6\baselineskip}
\section{Introduction}

Crystallization is one of the oldest and major processes used in industry (chemical, pharmaceutical, food, etc.) to produce, purify or separate solid compounds or products \cite{kleinboistelle}.
This unit operation aims to produce solid crystals with well defined specifications including (among others) the Crytal Size Distribution (CSD) which is of critical importance.
At the industrial scale, the CSD is neither well controlled nor monitored during the crystallization process and a grinding step is usually performed before delivering the final product.
Hence,
a key point is the real-time control and the ``\emph{online}" monitoring of the CSD during the crystallization operation in order to avoid the grinding step. 
Unfortunately,
the CSD is nowadays not directly measurable in real-time by existing sensors.
Nevertheless,
the Process Analytical Technologies (PATs) allow us to get access to real time information such as the solute concentration based on the Attenuated Total Reflectance Fourier Transform InfraRed spectroscopy (ATR-FTIR) and the Chord Length Distribution based on the Focused Beam Reflectance Measurement (\fbrmp).
\startlucas 
 \fbrm is commonly employed in pharmaceutical industries to detect online some process deviations.
 Note however that its use for online estimation of the CSD remains challenging \cite{Poblete,scheler2013ray}.

Obtaining
an online
CSD estimation from the solute concentration, temperature and a growth rate model only is an interesting problem which is the purpose of the present paper.
To obtain online state estimation for dynamical systems from measured outputs, control engineers usually employ asymptotic state  observers.
Designing state observer for complex dynamical systems is an active research area. Some studies are devoted to nonlinear dynamics (see \cite{Bernard_Book_2019} and  references therein). Others consider infinite dimensional systems (see for instance \cite{TW}).

Designing observers for a batch crystallization process has been addressed in recent years by several researchers (see for instance \cite{mesbah2011real, mesbah2011comparison,motz2008state, nagy2003robust, porru2017monitoring, Uccheddu, vissers2012model}). 
These studies are based 
on spatial discretization of the PDE (see \cite{porru2017monitoring} or \cite{vissers2012model})
or more frequently on moments analysis (see survey \cite{mesbah2011comparison}).
However, the moment based approaches suffer from several drawbacks: (i) the numerical moment values exhibit a very large difference in their order of magnitude \cite{Thevenin}. Consequently, a small numerical/experimental error has a significant impact on the moments estimation. The error will then be propagated during the computation of the moment transport equations \cite{Jakobsen};
(ii) the recovering of the CSD from a finite number of its moments is still an open area of research in mathematics and highly dependant of the moments quality \cite{Lebaz, Rohani}.
\stoplucas
\startlucas
Furthermore, the
\stoplucas
moment based observers depend on the knowledge of the nucleation rate which may be tricky to model.
\startlucas
All these
\stoplucas
reasons may explain the difficulty to develop 
an efficient algorithm following this route.


In this paper,
another approach is adopted to describe the CSD without using its moments.
Moreover,
a key point of the proposed approach is that no information is needed on the nucleation rate.

In the first part of the paper, we recall the \startlucas Kazantzis-Kravaris/Luenberger \stoplucas methodology for observer design which has been introduced in its original form in \cite{luenberger1964observing} for linear systems and adapted for nonlinear dynamics in \startlucas
\cite{andrieu2006existence}  and \stoplucas \cite{kazantzis1998nonlinear}.
\startlucas
Section~\ref{secappli}
\stoplucas
is devoted to the modeling of the batch crystallization process.
Finally some simulation results are presented
\startlucas
in Section~\ref{secsimu}
\stoplucas
which highlight the practical interest of the suggested methodology.

\section{Kazantzis-Kravaris/Luenberger observer methodology}
\label{secluen}

\startlucas
In control engineering, the algorithm which is employed to reconstruct online missing data on a partially measured dynamical process is named an observer.
Observer design for nonlinear dynamical systems is a very active research area and has been the subject of numerous studies in the past 40 years (see \cite{Bernard_Book_2019} for a recent survey on this topic).
\stoplucas
One very efficient way to design observers for finite-dimensional linear as well as nonlinear dynamical systems 
\startlucas is the Kazantzis-Karavis/Lunberger (KKL) observer methodology. It
\stoplucas
is an approach which follows Luenberger original idea in \cite{luenberger1964observing} where asymptotic observers for linear systems were introduced for the first time.
The approach of \cite{luenberger1964observing} is somehow different from the way \textit{Luenberger observers} are introduced nowadays (which follows Luenberger second paper on observers \cite{luenberger1971introduction}). 
It has been recently employed by Kazantzis and Kravaris in \cite{kazantzis1998nonlinear} to design local observers for finite-dimensional nonlinear dynamics based on Luenberger original idea in a local version. We hope it may also be fruitful in the context of distributed parameter systems.
In this section we adapt to the context of linear abstract Cauchy problem the \startlucas KKL \stoplucas observer methodology which
is a two-steps design procedure.
In the first step we consider the problem of estimating a function of the state.
The reconstruction of the overall state of the system is obtained in the second step by inverting this function.

\subsection{Step 1: reconstruction of a function of the state}

As shown in \cite{andrieu2006existence} and \cite{bernard2019luenberger}, it is always possible to exponentially estimate a function of the state, even for nonlinear finite-dimensional systems,
\startlucas
that will carry enough information about the state to estimate it in Step 2.
\stoplucas
In order to do so,
it is sufficient to introduce an auxiliary dynamical system fed by the measured output such that its solutions provide an estimation of this function of the state.

\startlucas
Let $\RR$ (resp. $\RR_+$, $\RR_-$ and $\RR^*$) denotes the set of real numbers (resp. non-negative, non-positive and non-zero real numbers).
Let $\XR$ be a Hilbert space. We denote $\mathcal L (\XR)$ the set of all endomorphisms of $\XR$ and $\mathcal L (\XR, \RR)$ the space of linear forms from $\XR$ to $\RR$. Consider the abstract Cauchy problem on $\XR$
\stoplucas
\begin{equation}\label{eq_AbsCauchyProb}
\dot \psi = \FR \psi, \quad\psi(0) = \psi_0,   
\end{equation}
where
$\FR: D(\FR)\subset\XR\rightarrow \XR$ is a linear operator which is the generator of a strongly continuous semi-group denoted $(T_t)_{t\in\RR_+}$ in $\mathcal L (\XR)$
\startlucas
and $\psi_0\in D(\FR)$.
\stoplucas
Let $\rho(\FR) = \{\lambda \in \CC \mid (\FR-\lambda \Id)^{-1}\in \mathcal L (\XR)\}$ denote  the resolvent set of $\FR$.
Moreover, consider a bounded output operator
\begin{equation}\label{eq_OutputAbstract}
y = \HR \psi    ,
\end{equation}
\startlucas
where $\HR\in\mathcal L (\XR, \RR)$ is bounded.
\stoplucas
Following the \startlucas KKL 
\stoplucas
methodology, we obtain the following proposition.
\begin{proposition}\label{propobs}
For all $\lambda$ in $\rho(\FR)\cap \startlucas\RR_-^*\stoplucas$, let $\TR_\lambda$ in $\mathcal L(\XR,\RR)$ be the operator defined as
$$
\TR_{\lambda}: \XR\ni \psi \mapsto \HR(\FR-\lambda \Id)^{-1}\psi\in \RR\ .
$$
Then, the dynamical system
\begin{equation}\label{eq_Obs}
\dot z_\lambda = \lambda z_\lambda + y,
\end{equation}
is an
\startlucas
exponential
\stoplucas
observer for $\TR_{\lambda}\psi$. More precisely, for all $(\psi_0,z_0)$ in $D(\FR)\times\RR$, it yields for all $t\geq0$
\startlucas
\begin{align}
\TR_{\lambda}(T_t \psi_0)- z_\lambda(t) = \exp(\lambda t)\left(\TR_{\lambda}(\psi_0)- z_0\right).
\label{eq_vitconv}
\end{align}
\stoplucas
where $z_\lambda:\RR_+\rightarrow\RR$ is the solution of system (\ref{eq_Obs}) when $y$ is given by (\ref{eq_OutputAbstract}) and initiated from $z_0$.
\end{proposition}
\begin{proof}
Let $\psi_0$ be in $D(\FR)$.
Equations \eqref{eq_AbsCauchyProb}--\eqref{eq_Obs} yield
\begin{align*}
\frac{d}{dt}\big(\TR_{\lambda}(T_t \psi_0)- z_\lambda(t)\big)
&=  \TR_{\lambda}(\FR T_t \psi_0) - \lambda z_\lambda(t) - \HR T_t \psi_0
\\
&=  \TR_{\lambda}(\FR - \lambda \Id) T_t \psi_0  +  \lambda \big(\TR_{\lambda}(T_t \psi_0)  -  z_\lambda(t) \big) - \HR T_t \psi_0
\\
&= \lambda (\TR_{\lambda}(T_t\psi_0) -  z_\lambda(t)),
\end{align*}
where the last equality follows since $\TR_{\lambda}(\FR - \lambda \Id) = \HR$.
\startlucas
Hence, \eqref{eq_vitconv} follows by integrating in time the former equation.
\stoplucas

\startlucas
Keeping in mind that $\lambda$ is negative in Proposition~\ref{propobs}, \eqref{eq_vitconv} implies
\begin{equation}\label{eq_Conv}
\lim_{t\rightarrow +\infty} \left|\TR_{\lambda}(T_t \psi_0)- z_\lambda(t)\right| = 0 .
\end{equation}
This ends the proof.
\stoplucas
\end{proof}
\startlucas
\begin{remark}
The operator $\TR_{\lambda}$ is solution to the Sylvester equation~:
\begin{equation}
\FR\TR_{\lambda} = \lambda \TR_\lambda + \HR .
\end{equation}
We recognize here the algebraic equation which was already given in Luenberger seminal paper \cite{luenberger1964observing} and which becomes a nonlinear partial differential equation in \cite{andrieu2006existence}.
\end{remark}
\stoplucas
\subsection{Step 2: reconstruction of the entire state of the system}
According to step 1, we can easily estimate $\TR_{\lambda}\psi$ for all $\lambda$ in $\rho(\FR)\cap \startlucas\RR_-^*\stoplucas$ via the observer system (\ref{eq_Obs}).
The idea of the \startlucas KKL 
\stoplucas
observer methodology is to consider the mapping
\(
\TR:\XR\mapsto \RR^p
\)
given by $\psi \mapsto (\TR_{\lambda_1} \psi, \dots, \TR_{\lambda_p} \psi)$
which will be
\startlucas
exponentially
\stoplucas
estimated along the trajectory of \eqref{eq_AbsCauchyProb} via a bench of observers of the form (\ref{eq_Obs}).
To solve the estimation problem, the question is to solve the inverse problem 
\begin{equation}\label{E:equality-constraint}
\TR \hat \psi = z
\end{equation}
with the unknown  $\hat \psi$ in $\XR$.
\startlucas
Let $\Im\TR=\{\TR x\mid x\in\XR\}$ be the image of $\TR$.
Then \eqref{E:equality-constraint} admits a solution only if $z$ is in $\Im\TR$.
\stoplucas
The former condition is in general too restrictive due to the fact that $z$ is only an estimation of $\TR \psi$ and consequently may not be in $\Im\TR$.
A solution to overcome this problem is to replace the equality constraint \eqref{E:equality-constraint} by the minimization problem
\begin{equation}\label{pb:min_ill}
\textrm{minimize~}\|\TR \psi - z\|^2 \textrm{~subject to~} \psi \in \XR.
\end{equation}
The set of solutions of \eqref{pb:min_ill} is denoted by
\(
\argmin_{\psi \in \XR}\|\TR \psi - z\|^2\).
The following results can be found in \cite[Chapter 4]{kern2016methodes}:
\begin{itemize}
    \item If $\TR$ is injective, then \eqref{pb:min_ill} has at most one solution.
    \item If $z\in \Im\TR \oplus \left(\Im\TR\right)^\perp$, then the set $\argmin_{\psi \in \XR}\|\TR \psi - z\|^2$ is closed, convex and non-empty (in particular \eqref{pb:min_ill} admits at least one solution).
    \item If $\TR$ is bijective and admits a left inverse denoted $\TR^{-1}$ then the unique solution of \eqref{pb:min_ill} is $\hat \psi(t)=\TR^{-1}z(t)$.
\end{itemize}

For finite dimensional systems, the injectivity of $\TR$ is directly linked with an observability property of the dynamical system \eqref{eq_AbsCauchyProb}.
Indeed, following \cite{luenberger1964observing}, it can be shown that for finite dimensional systems, 
if the pair $(\FR,\HR)$ is observable (in the sense that the Kalman observability matrix is full rank), then picking $p$ pairwise distinct $\lambda_i$'s with $p = \dim \XR$, then the obtained mapping $\TR$ is invertible.
In that case, a \startlucas KKL 
\stoplucas
 observer can simply be obtained as
$$
\dot z =\begin{bmatrix}\lambda_1 & & \\
&\ddots&\\
&&\lambda_p\end{bmatrix}  z +\begin{bmatrix}1\\\vdots\\1\end{bmatrix} y\ ,\quad \hat \psi(t) = \TR^{-1} z(t).
$$

In the infinite dimensional case, there is no hope that a finite number of $\lambda_i$'s can give all the information allowing to reconstruct the entire system's state. It is an open question to know if observability properties of the infinite dimensional system allows to obtain that a certain countable set of $\lambda_i$'s gives enough information.
In that case, the minimization problem \eqref{pb:min_ill} becomes ill-posed.
\startlucas
In other words, it may have no solution, or numerous solutions, and its solutions may depend on $z$ in a non-continuous manner.
In our case, $z$ is an estimation of $\TR\psi$. Hence, the minimizer $\hat\psi$ of \eqref{pb:min_ill} may be very different of the real state $\psi$.
A typical approach to overcome this problem is to consider a Tikhonov regularization method in which the optimization problem \eqref{pb:min_ill} is slightly modified.
This new minimization problem will have solutions close to the former one, and be well-posed.
We recall the next proposition on which this method is based for the convenience of the reader.
\stoplucas
\begin{proposition}[\cite{kern2016methodes}, Proposition \startlucas 6\stoplucas.1]
Given $\TR$ in $\mathcal L(\XR,\RR^p)$ and $\delta>0$, 
the minimization problem
\begin{equation}\label{pb:min_well}
\textrm{minimize~} \|\TR \psi - z\|^2+ \delta \|\psi \|^2 \textrm{~subject to~} \psi \in \XR
\end{equation}
admits a unique solution and is a continuous function of $z$ in $\RR^p$.
\end{proposition}


The choice of the \emph{regularization parameter} $\delta$ is of the uttermost importance. Indeed, as $\delta\to+\infty$, the solution of \eqref{pb:min_well} goes to zero. Conversely, as $\delta\to0$, the problem \eqref{pb:min_well} goes closer to the ill-posed problem \eqref{pb:min_ill}. A compromise shall be made. In practice, $\delta$ is chosen experimentally, and is linked to the measure confidence: the more uncertain is the output $y$, the bigger is $\delta$.

We also have the following theoretical result, that describes what happens when $\delta$ goes to zero.

\begin{theorem}[\protect{\cite[Theorem \startlucas 6\stoplucas.1]{kern2016methodes}}]
Let $\TR$ in $\mathcal L(\XR,\RR^p)$ and $z \in \Im \TR$. Let $\psi_0\in \XR$ and $\psi$ the solution of \eqref{pb:min_ill} closest to $\psi_0$. Let $(z_n)_{n\in \NN}$ be a sequence in $\RR^p$ converging to $z$.
Let $\varepsilon_n = |z_n-z|$. Let $(\delta_n)_{n\in \NN}$ be a sequence of regularization parameters converging to zero. For any $n\in \NN$, let $\psi_n$ be the solution of the problem \eqref{pb:min_well} associated to $z_n$ and $\delta_n$. Then,
\begin{itemize}
\item $\displaystyle|\TR \psi_n-z_n|\underset{n\rightarrow+\infty}{\longrightarrow}0$;
\item
if
\(\displaystyle
\frac{\varepsilon_n}{\delta_n}\underset{n\rightarrow+\infty}{\longrightarrow}0,
\)
then
\(|\TR\psi_n-z_n| = \OO(\varepsilon_n)\) and \(
\psi_n\underset{n\rightarrow+\infty}{\longrightarrow}\psi;
\)
\item
if
\(\displaystyle
\frac{\varepsilon_n}{\delta_n^2}\underset{n\rightarrow+\infty}{\longrightarrow}0
\)
and
\(
\psi\in\left(\Im\TR\right)^*,
\)
then
\(
|\TR\psi_n-z_n| = \OO(\varepsilon_n^2)
\)
and
\(
|\psi_n-\psi| = \OO(\varepsilon_n).
\)
%
\end{itemize}
\end{theorem}
\startlucas
To summarize, a possible observer design of a given abstract Cauchy problem in the form \eqref{eq_AbsCauchyProb} is then given by
\begin{equation}\label{eq_Obs_KKL_Tikhonov}
\left\{
\begin{aligned}
&\dot z =\begin{bmatrix}\lambda_1 & & \\
&\ddots&\\
&&\lambda_p\end{bmatrix}  z +\begin{bmatrix}1\\\vdots\\1\end{bmatrix} y
\\
&\hat \psi(t)
= \argmin_{\psi \in \XR}\left\{\|\TR \psi - z(t)\|^2
+ \delta \|\psi\|^2\right\}, \quad \delta >0
\\
&\TR
= \left(\TR_{\lambda_1}, \dots, \TR_{\lambda_p}\right), \quad \TR_{\lambda_i}  = \HR(\FR - \lambda_i \Id)^{-1}, \quad
i=1,\dots, p,
\end{aligned}\right.
\end{equation}
where the $\lambda_i$ are  pairwise distinct elements of $\rho(\FR)\cap \startlucas\RR_-^*\stoplucas$.

Note that despite the fact that this dynamical system is well-defined for all abstract Cauchy problem in the form \eqref{eq_AbsCauchyProb}, its convergence to the real state is not guaranteed \textit{a priori} and may be linked to observability properties.

\stoplucas

\section{Application to the crystallization process}\label{secappli}

\subsection{Modeling the batch crystallization process}

In this subsection, a dynamical model representing the batch crystallization process and the concentration measurement is given.

\subsubsection{Population balance equation}\label{secbilan}

In a first step a batch crystallization process is modeled.
\startlucas
We assume that the size of the crystals is described by a scalar parameter $x$ (in $m$).
For example, if the crystals are spherical, then $x$ may represent their diameter.
\stoplucas

Let us denote $\etat(t, \cdot)$
the Number Density Function in terms of the crystal size (NDF) at time $t$ \startlucas
(in $\mathrm{m}^{-1}.\mathrm{m}^{-3}$),
\stoplucas
so that
$\int_{x_1}^{x_2}\etat(t, x)dx$ is the total amount of crystals in the reactor at time $t$ with size between $x_1$ and $x_2$
\startlucas
(in $\mathrm{m}^{-3}$).
\stoplucas
Let  $[t_0, t_1]$ be the time window in which the crystallization process occurs.

We assume that the crystals never reach a specific maximal allowable size
denoted  $\xmax$ during the experiment.
Typically this size can simply be related to the size of the reactor in which the crystallization occurs:
\begin{equation}\label{eq_EtatMax}
\forall t \in [t_0, t_1],\quad \etat(t, \xmax) = 0.
\end{equation}
We assume that all the crystals appear at the same positive size $\xmin$,  and we denote $\cont(t)$ the appearance rate of new crystals at size $\xmin$ at time $t$.
\startlucas
The function $\cont$ quantifies the nucleation rate in the reactor.
In \cite{nagy2013recent}, for instance, an expression for the function $\cont$ is given as
\(
\cont(t) = {R_n(t)}/{\vit(t)} , 
\)
where $R_n$ is the overall rate of nucleation expressed in $\mathrm{s^{-1}.m^{-3}}$ and $G$ is the growth rate of the crystals in $\mathrm{m.s^{-1}}$.
Note however that in our approach for NDF observation, 
we don't need to know precisely this expression. We do not use any model of $u$, and assume this quantity to be unknown.
\stoplucas
We have
\startlucas
\begin{align}
\forall t \in [t_0, t_1],\quad \etat(t, \xmin) = \cont(t).
\end{align}
\stoplucas
Finally, let  $\vit(t, x) \geq 0$ be the growth rate of the crystals in $\mathrm{m.s^{-1}}$, in other words the rate at which a crystal of size $x$ grows at time $t$. 
The population balance
\startlucas
leads to
\stoplucas
\begin{equation}
\forall t \in (t_0, t_1),\quad \forall x\in (\xmin, \xmax),\quad \partial_t \etat(t, x) + \partial_x (\vit \etat)(t, x) = 0.
\label{trans1}
\end{equation}
By considering the McCabe assumption, we assume that $\vit$ does not depend on $x$. 
In that case equation \eqref{trans1} becomes
\begin{equation}
\forall t \in (t_0, t_1),\quad \forall x\in (\xmin, \xmax),\quad \partial_t\etat(t, x) + \vit(t)\partial_x\etat(t, x) = 0.
\label{trans}
\end{equation}
Equation (\ref{trans}) is a time-varying one-dimensional transport equation. 

We also assume that at the beginning of the experiment, some seed particles are in the reactor.
This yields
\begin{align*}
\forall x \in [\xmin, \xmax], \quad \etat(t_0, x) = \etat_0(x).
\end{align*}

To summarize, the model of the NDF in a batch crystallization process is:
\begin{equation}
\begin{aligned}
\begin{cases}
\partial_t \etat(t, x) = -\vit(t)\partial_x \etat(t, x) & \forall t \in (t_0, t_1), \ \forall x \in (\xmin, \xmax)\\
\etat(t_0, x) = \etat_0(x) & \forall x \in [\xmin, \xmax]\\
\etat(t, \xmin) = \cont(t) &\forall t \in [t_0, t_1]
\end{cases}
\end{aligned}
\label{systbilan}
\end{equation}
to which is added the other boundary condition in (\ref{eq_EtatMax}) which in our case is seen as a knowledge on the particular solution we wish to estimate.

The following theorem states that system \eqref{systbilan} of this model admits weak solution in $L^2$ space and strong solution in $H^1$.
\begin{theorem}\label{th:exist}
Let $t_1 > t_0 \geq 0,\ \xmax > \xmin\geq 0,\ \etat_0 \in H^1(\xmin, \xmax),\ \cont \in H^1(t_0, t_1),\ \vit \in C^0([t_0, t_1]; \RR^*_+)$. Assume that $\cont(t_0) = \etat_0(\xmin)$. Then system \eqref{systbilan} admits a unique solution
\begin{align*}
\etat \in C^0([t_0, t_1]; H^1(\xmin, \xmax)) \cap C^1([t_0, t_1]; L^2(\xmin, \xmax)).
\end{align*}
Moreover, for all $(t, x) \in [t_0, t_1] \times [\xmin, \xmax]$,
\begin{align}
\etat(t, x)=
\begin{cases}
\etat_0 (x - \intg(t)) &\text{ if } x - \xmin \geq \intg(t) \\
\cont \circ \intg^{-1} (\intg(t) - x + \xmin) &\text{ else.}
\end{cases}
\label{eqf}
\end{align}
where $\intg: [t_0, t_1] \ni t \mapsto \int_{t_0}^t\vit(\tau)\dd \tau$.

\end{theorem}
The proof of this theorem can be found in \cite[Theorem 2.4]{Coron} in the case $\vit = 1$, and can be easily adapted by means of a time reparametrization.
It is worth noticing that this theorem does not take into account hypothesis \eqref{eq_EtatMax}.
However, the following proposition holds.
\begin{proposition}
Assume that the hypothesis of Theorem \ref{th:exist} are satisfied.
Assume that there exists $\bar{x}\in[\xmin, \xmax)$ such that $\etat_0(x) = 0$ for all $x\in[\bar{x}, \xmax]$.
If
\begin{align}
    \bar{x} + \intg(t_1) < \xmax,
\end{align}
then $\psi(t, x) = 0$ for all $t\in[t_0, t_1]$ and all $x\in\left[\bar{x} + \intg(t_1), \xmax \right]$.
\end{proposition}
\begin{proof}
Let $t\in[t_0, t_1]$ and $x\in\left[\bar{x} + \intg(t_1), \xmax \right]$. Then
\begin{align*}
    x - \intg(t_1) \geq \bar{x} \geq \xmin.
\end{align*}
Consequently, according to \eqref{eqf},
\(
    \etat(t, x) = \etat_0(x - \intg(t)) = 0.
\)
\end{proof}

Hence, one must choose $t_1$ small enough so that the particles did not reach the size $\xmax$. Roughly speaking, this means that the observer that we are going to design must estimate the state in small time, \emph{i.e.} before the particles reach the size $\xmax$.
In the following, we always assume that \eqref{eq_EtatMax} is satisfied.

\subsection{Concentration sensor modeling}

In the considered batch crystallization process, the measured outputs are the temperature and the solute concentration denoted $\cste(t)$.
These two measurements allow to obtain online estimation of the growth rate (\emph{i.e.} $\vit$) and the third moment of the NDF.

\subsubsection{\texorpdfstring{Estimation of $\vit$}{Estimation of G}}
\label{estG}

The knowledge of the temperature and the solute concentration allows to obtain some approximation of the growth rate $\vit$ .
Indeed, following \cite{Uccheddu}, a model of $\vit$ can be given for all time $t\in [t_0, t_1]$ by
\begin{equation}
\vit(t) = k_g \frac{\cste(t) - C^*(t)}{C^*(t)}
\label{eqG}
\end{equation}
where
\begin{itemize}
\item $k_g$ is a known growth rate parameter (in $\mathrm{m.s^{-1}}$),
\item $C^*(t)$ is the solubility at time $t$ (in $\mathrm{kg}$ of solute per $\mathrm{kg}$ of solvent),
\item $\cste(t)$ is the solute concentration at time $t$ (in $\mathrm{kg}$ of solute per $\mathrm{kg}$ of solvent).
\end{itemize}
Since $C^*(t)$ depends on the temperature at time $t$ ,
the growth rate $\vit$ of the crystals can be estimated online with the available sensors.\\
\startlucas
Other model expressions of $G$ are available in the literature, for more details one may refer to \cite{Mersmann,Mullin}.
\stoplucas

\subsubsection{Estimation of the third moment of the NDF}
It is possible to link the solute concentration with the NDF.
Indeed, for each $t\in \RR_+$, let $\cs(t)$ (in $\mathrm{kg}$ of solid per $\mathrm{kg}$ of solvent) be the solid concentration in the reactor at time $t$, 
in other words, the ratio between the total crystals mass in the reactor at time $t$ and the solvent mass. 
Let $\rho_s$ (in $\mathrm{kg}$.$\mathrm{m^{-3}}$) be the density of the solute in solid phase and $M_e$ the solvent mass (in $\mathrm{kg}$).
It yields:
\begin{align*}
\cs(t) = \frac{\rho_s}{M_e}V_s(t)
\end{align*}
where $V_s(t)$ is the volume (in $\mathrm{m^3}$) occupied by the crystals at time $t$.
\startlucas
Then the volume of a crystal with size $x$ (in $\mathrm{m}$) is simply $V = k_v x^3$
where $k_v$ is a volumetric shape factor (see \emph{e.g.} \cite{Hulburtkatz,Randolph}). For example, $k_v = \pi/6$ for spherical crystals.
The total volume of the crystals is then
\stoplucas
\begin{align*}
V_s(t) = k_v\int_{\xmin}^{\xmax} \etat(t, x) x^3 \dd x
\end{align*}
Hence, the solid concentration in the reactor can be expressed as follows.
\begin{align}
\forall t \in [t_0, t_1],\quad \cs(t) = \frac{\rho_s k_v}{M_e}\int_{\xmin}^{\xmax} \etat(t, x) x^3 \dd x.
\end{align}

Assume moreover that $\rho_s$ is a known parameter.
This implies that we can associate  to system \eqref{systbilan}  the measurement $\mes$ defined as
\begin{equation}
\forall t \in [t_0, t_1],\quad \mes(t) = \int_{\xmin}^{\xmax} \etat(t, x) x^3 \dd x.
\label{eqy}
\end{equation}

From there, the observation problem we intend to solve is the following.
\textit{From the knowledge of the output function $y(t)$ and the growth rate, give an online estimation of the NDF.}

The purpose of the  Section \ref{Sec_Obs_Cryst} is to propose a novel algorithm to solve this problem based on the
\startlucas KKL 
\stoplucas
 methodology which was described in Section \ref{secluen}.
Note however that before implementing the observer an observability analysis can be carried out.

\subsection{About the observability of the crystallization model}

In this section, we study how the third moment may help us to estimate the NDF.
First, we have the following result.
\begin{proposition}\label{prop_pos}
Let $\tau\in(t_0, t_1]$.
\startlucas
Assume that there exists $\mu>0$ such that $\vit(t)\geq\mu$ for all $t\in[t_0, t_1]$.
\stoplucas
Then for all $\mes\in C^0(t_0, \tau)$, there exists at most one function $\cont\in H^4(t_0, \tau)$ such that the solution $\etat$ of \eqref{systbilan} given by $\cont$ and $\etat_0 = 0$ satisfies
\begin{align*}
    \mes(t) = \int_{\xmin}^{\xmax} \etat(t, x)x^3\dd x,\qquad \forall t\in[t_0, \tau].
\end{align*}
\end{proposition}
In other words,
Proposition \ref{prop_pos} states that the map $\cont\mapsto \mes$ is injective, where $\mes$ denotes the third moment of the solution of \eqref{systbilan} with null initial condition.
Its proof can be found in appendix.
Hence, one can hope that our method may reconstruct $\etat$ from $\mes$, at least when the initial condition is zero (\emph{i.e.} there is no crystals at the beginning of the process).

However,
one can wonder what happens if the initial condition is not zero. Can we still reconstruct the state from the measurement of its third moment and the knowledge of its dynamics? In other words, is the map $\etat\mapsto \mes$ injective?
If yes,
then one can hope that our algorithm is robust, so that the estimation of the state converges to the actual NDF.
\startlucas
Unfortunately,
the answer is no. Indeed, we have the following proposition which is a slight modification of \cite[Theorem 3.2.3]{vissers2012model} that we state in our own context only. For the convenience of the reader, its proof is given in appendix.
\stoplucas

\begin{proposition}
\label{prop_neg}
Let $\tau\in(t_0, t_1]$.
\startlucas
Assume that there exists $\mu>0$ such that $\vit(t)\geq\mu$ for all $t\in[t_0, t_1]$.
\stoplucas
There exist infinitely many solutions of \eqref{systbilan} with different initial conditions and boundary conditions that have the same third moment $\mes\in C^0(t_0, \tau)$.
\end{proposition}
We shall say that system \eqref{systbilan} with measurement \eqref{eqy} is not observable.
\startlucas
Thus, we cannot guarantee that our estimation of the NDF converges to the actual NDF.
Despite this fact, our methodology should be able to reconstruct partially the actual NDF.
Indeed, the linear function that maps the NDF to its third moment has rank 4.
The image of this state-output mapping is sometimes called the \emph{observable part} of the system (see \emph{e.g.} \cite{haine2014recovering}), due to the fact that an observer shall estimate at least the projection of the actual state on this subspace.{}
See the proof in appendix for more details.
\stoplucas

\subsection{A dynamical observer from the concentration measurement}
\label{Sec_Obs_Cryst}

Eventhough it was shown in the former section that the system is not observable, in this subsection, we show how the 
\startlucas KKL 
\stoplucas
observer approach can be employed on the considered model.
Following the procedure given in Section \ref{secluen}, we
consider $\lambda$ a negative real number and the dynamical system
\begin{equation}\label{eq_ObsCrist}
\dot z = \lambda z + y\ .
\end{equation}
We must find a mapping $\TR_\lambda$ which is estimated by this dynamical equation.
We have the following proposition.
\begin{proposition}
Let $\TR_\lambda : C^1([t_0,t_1]; L^2(\xmin, \xmax))\mapsto C^1([t_0,t_1])$ be the functional defined as
\begin{equation}\label{defT}
    \TR_\lambda(\etat) : t \mapsto \int_{\xmin}^{\xmax}a(t,x)\etat(t,x)\dd x
\end{equation}
where
\begin{equation}
\begin{cases}
\partial_t a(t, x) + \vit(t) \partial_x a(t, x) = \lambda a(t, x) + x^3 & \forall t \in (t_0, t_1), \ \forall x \in (\xmin, \xmax)\\
a(t_0, x) = 0 &\forall x \in [\xmin, \xmax]\\
a(t, \xmin) = 0 &\forall t \in [t_0, t_1],
\end{cases}
\label{edpa}
\end{equation}
then, along the solution of \eqref{eq_ObsCrist} which satisfies \eqref{eq_EtatMax} where $y$ is given in \eqref{eqy}, it yields for all $t$ in $[t_0,t_1]$
\begin{equation}\label{eq_ObTCrist}
\startlucas
\TR_\lambda(\etat)(t) - z(t) = \exp(\lambda t) (\TR_\lambda(\etat)(0) - z_0).
\stoplucas
\end{equation}
\end{proposition}
\begin{proof}
Using \eqref{systbilan} and an integration by parts yields
\begin{align*}
\frac{d}{dt}(\TR_\lambda(\etat)(t) - z(t))
&= \int_{\xmin}^{\xmax}\partial_t a(t, x) \etat(t,x) dx
    - \int_{\xmin}^{\xmax}\vit(t)  a(t,x) \partial_x \etat(t,x)  dx
\\
& \qquad -\lambda z(t) -\int_{\xmin}^{\xmax}x^3 \etat(t,x) dx
\\
&= \int_{\xmin}^{\xmax}\partial_t a(t, x) \etat(t,x) dx
    + \int_{\xmin}^{\xmax}\vit(t)  \partial_x a(t,x)  \etat(t,x)  dx \\
& \qquad - G(t)\big[a(t,x)\etat(t,x) \big]_{\xmin}^{\xmax}
    -\lambda z(t) -  \int_{\xmin}^{\xmax}x^3 \etat(t,x) dx .
\end{align*}
Hence, with \eqref{edpa} and also the boundary condition in  \eqref{systbilan} and \eqref{eq_EtatMax}, this implies
$$
\frac{d}{dt}(\TR_\lambda(\etat)(t) - z(t)) =  \lambda (\TR_\lambda(\etat)(t) - z(t)).
$$
By integrating in time the former equation, we obtain \eqref{eq_ObTCrist}.
\end{proof}

Consequently, for each negative $\lambda$ we exponentially estimate the functional $\TR_\lambda \etat(t)$.
It is interesting to remark that no information on the nucleation rate is needed to obtain this estimation.
The state observer is given as \eqref{eq_Obs_KKL_Tikhonov}, after a choice of the regularization parameter $\delta$.

\section{Numerical simulations}\label{secsimu}
In this subsection
\startlucas
numerical simulations are
\stoplucas
carried out.
\startlucas
Let $(\mathbb{x}_j)_{1\leq j\leq N_x}$ be a uniform discretization of the space interval $(x_{\min}, x_{\max})$ with space step $\Delta x$ and $(\mathbb{t}_k)_{1\leq k\leq N_t}$ be a uniform discretization of the time interval $(t_0, t_1)$ with time step $\Delta t$.
We fix $N_x = N_t = 100$.
Let $(\lambda_j)_{1\leq i\leq p}$ be the considered negative values of $\lambda$.
An approximation of $(\TR_{\lambda_i}\psi)(\mathbb{t}_k)$ is given by $\Delta x \sum_{j=1}^{N_x} a_{i, j, k} \psi_{j, k}$ where $a_{i, j, k}$ is an approximation of $a_{\lambda_i}(\mathbb{t}_k, \mathbb{x}_j)$ (solution of \eqref{edpa}) and $\psi_{j, k}$ an approximation of $\psi(\mathbb{t}_k, \mathbb{x}_j)$.
\stoplucas
The transport equation which describes the crystallization process is simulated via the method of characteristics.

\startlucas
We consider the system \eqref{systbilan} with $\vit$ as in \eqref{eqG} with a null initial condition $\etat_0 = 0$ and a boundary condition similar to a truncated normal distribution reaching its maximum at $t=3$ $\mathrm{s}$ and with a compact support $[0, 6]$
\startlucas
(see Fig.~\ref{fig:u}).
\stoplucas
The unique solution of this system is drawn in Fig.~\ref{fig:sol} (solid line),
\startlucas
and the corresponding growth rate is drawn in Fig.~\ref{fig:G}.
\stoplucas

\begin{figure}
    \begin{subfigure}{.5\linewidth}
    \centering
    \includegraphics[width = \linewidth]{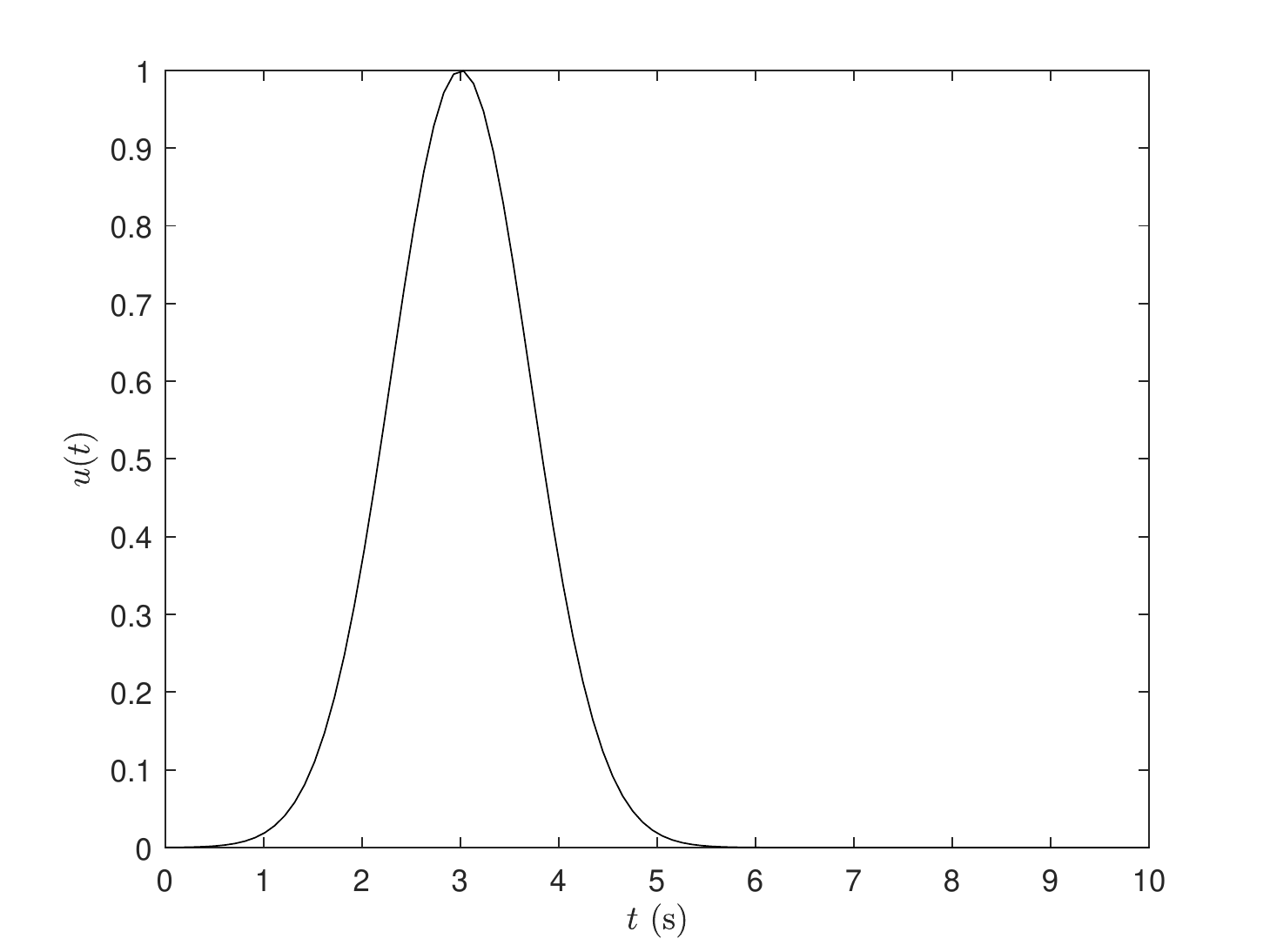}
    \caption{Boundary condition $u$ due to the nucleation}
    \label{fig:u}
    \end{subfigure}
    \begin{subfigure}{.5\linewidth}
    \centering
    \includegraphics[width = \linewidth]{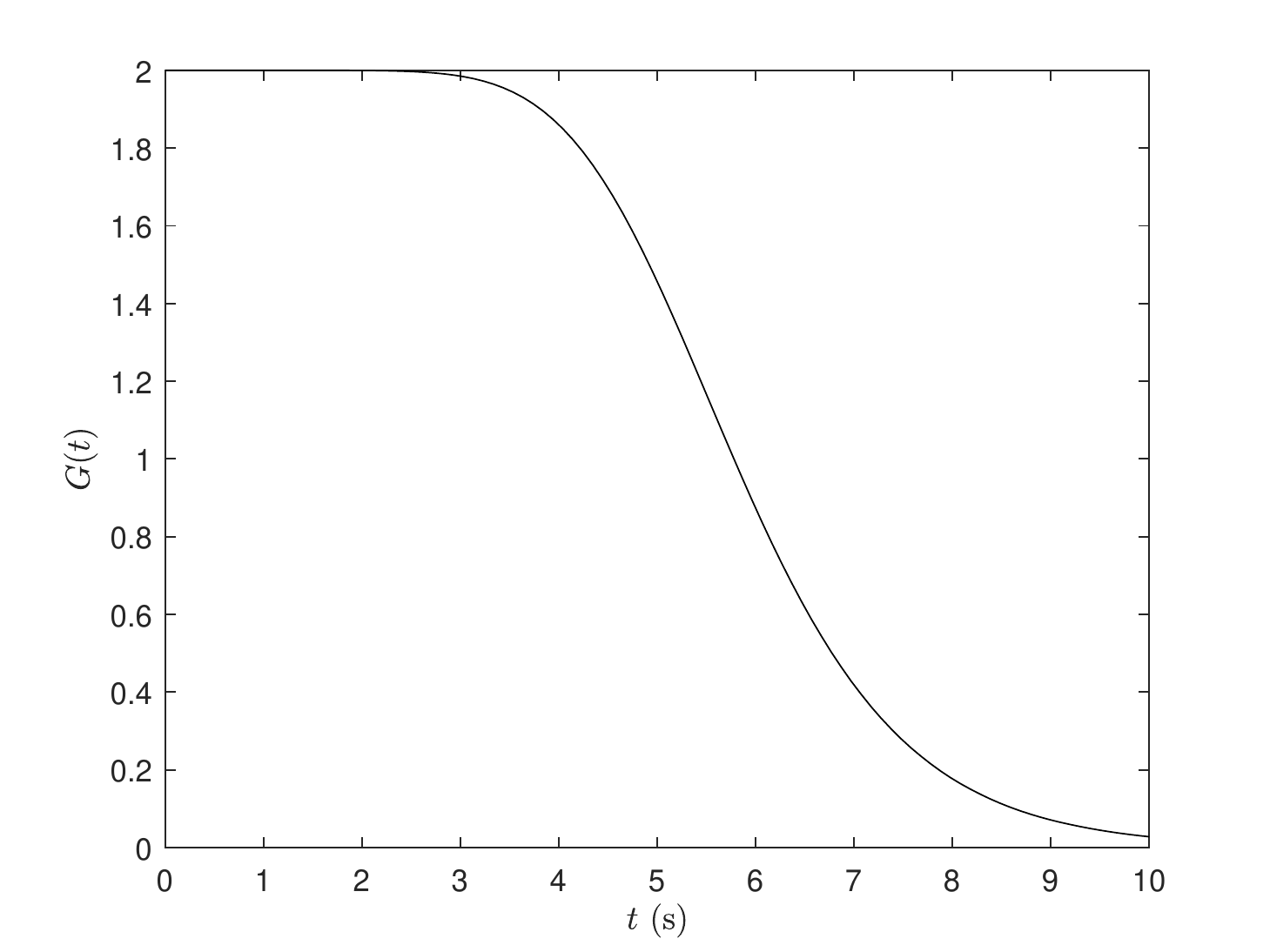}
    \caption{Growth rate $\vit$, given by \eqref{eqG}}
    \label{fig:G}
    \end{subfigure}\par\medskip
    \centering
    \begin{subfigure}{.5\linewidth}
    \centering
    \includegraphics[width = \linewidth]{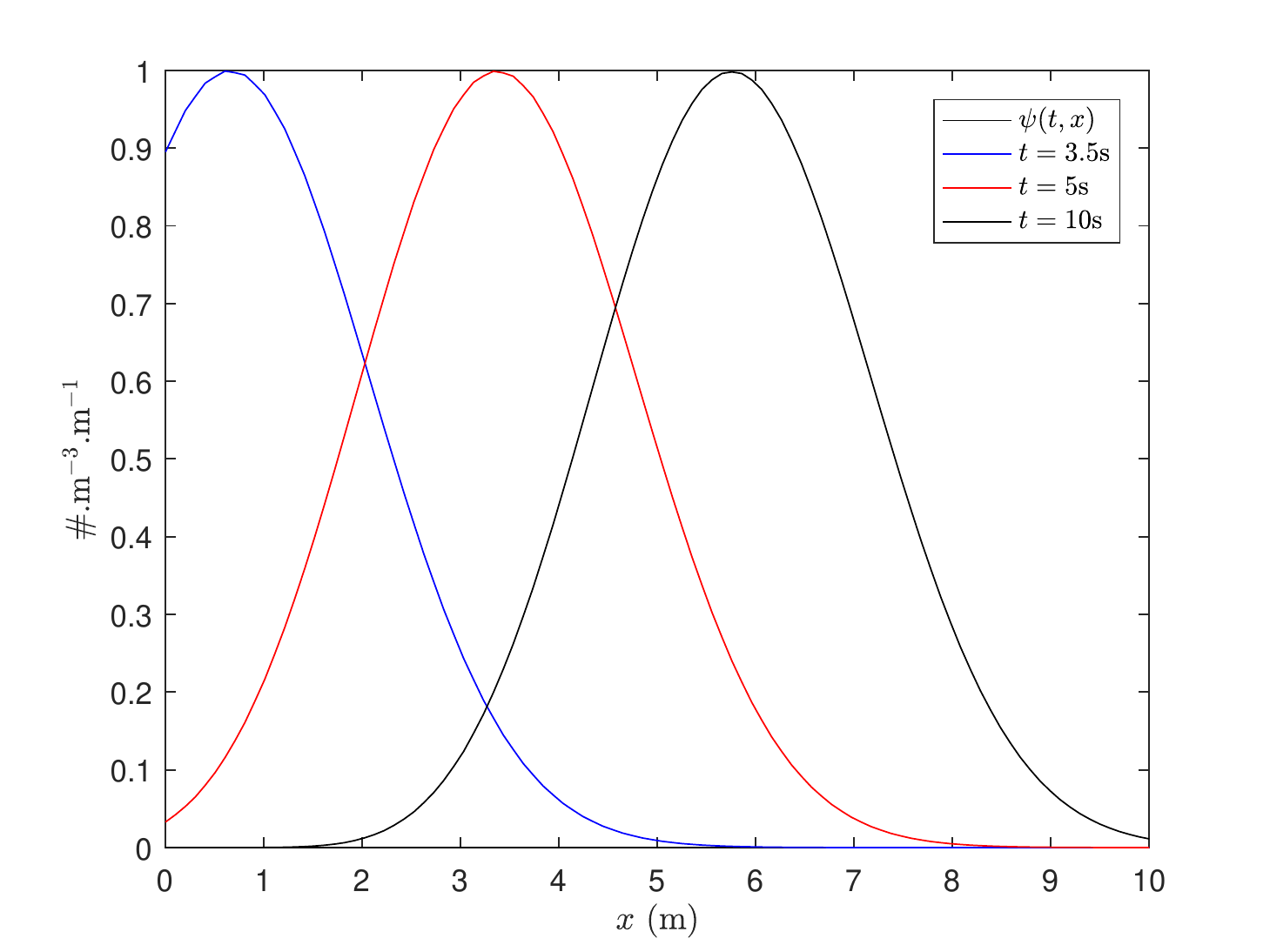}
    \caption{Simulation of the NDF $\psi$}
    \label{fig:sol}
    \end{subfigure}
    \caption{Numerical simulation of the process with $(\xmin, \xmax) = (0, 10)$, $(t_0, t_1) = (0, 10)$ and $N_x = N_t = 100$.}
    \label{fig:all_sim}
\end{figure}




\subsection{Step 1: reconstruction of a function of the state}

Following the methodology developed in Section \ref{secluen}, we first try to estimate the function $\TR_\lambda(\etat)$ of the state via the dynamical system \eqref{eq_ObsCrist} for some fixed negative values of $\lambda$.
\startlucas
All along the simulation of \eqref{systbilan}, we compute $\mes$ and estimate the solution of \eqref{edpa} via the method of characteristics.
\stoplucas
We integrate the solution of \eqref{eq_ObsCrist} with the first order Euler's method.
\stoplucas
Then we plot the evolution of the relative error between $z$ and $\TR_\lambda(\etat)$ in Fig.~\ref{fig:func_state} for some values of $\lambda$. One can check that the error goes to zero as $t\to+\infty$. Moreover, the bigger is $|\lambda|$, the faster is the convergence. This is due to the exponential convergence of $z-\TR_\lambda(\etat)$ to zero given by \eqref{eq_vitconv}.
Hence, we are able to approximate any function $\TR_\lambda(\etat)$ of the state.
Now, we can move to the second part of the methodology of Section \ref{secluen}.
\begin{figure}[htbp]
\centering
\begin{tabular}{r@{}l}
\includegraphics[width = 0.5\textwidth]{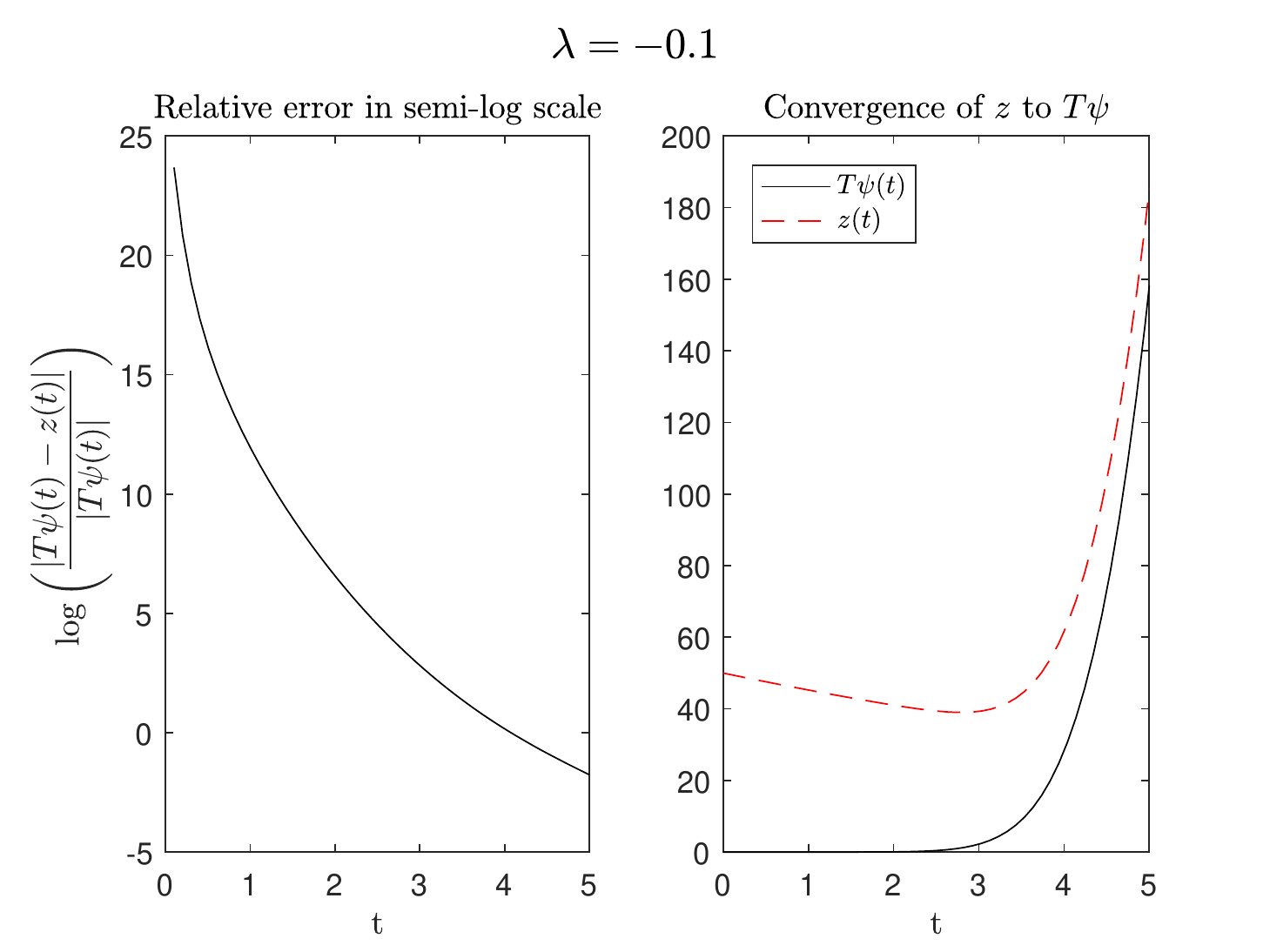}&
\includegraphics[width = 0.5\textwidth]{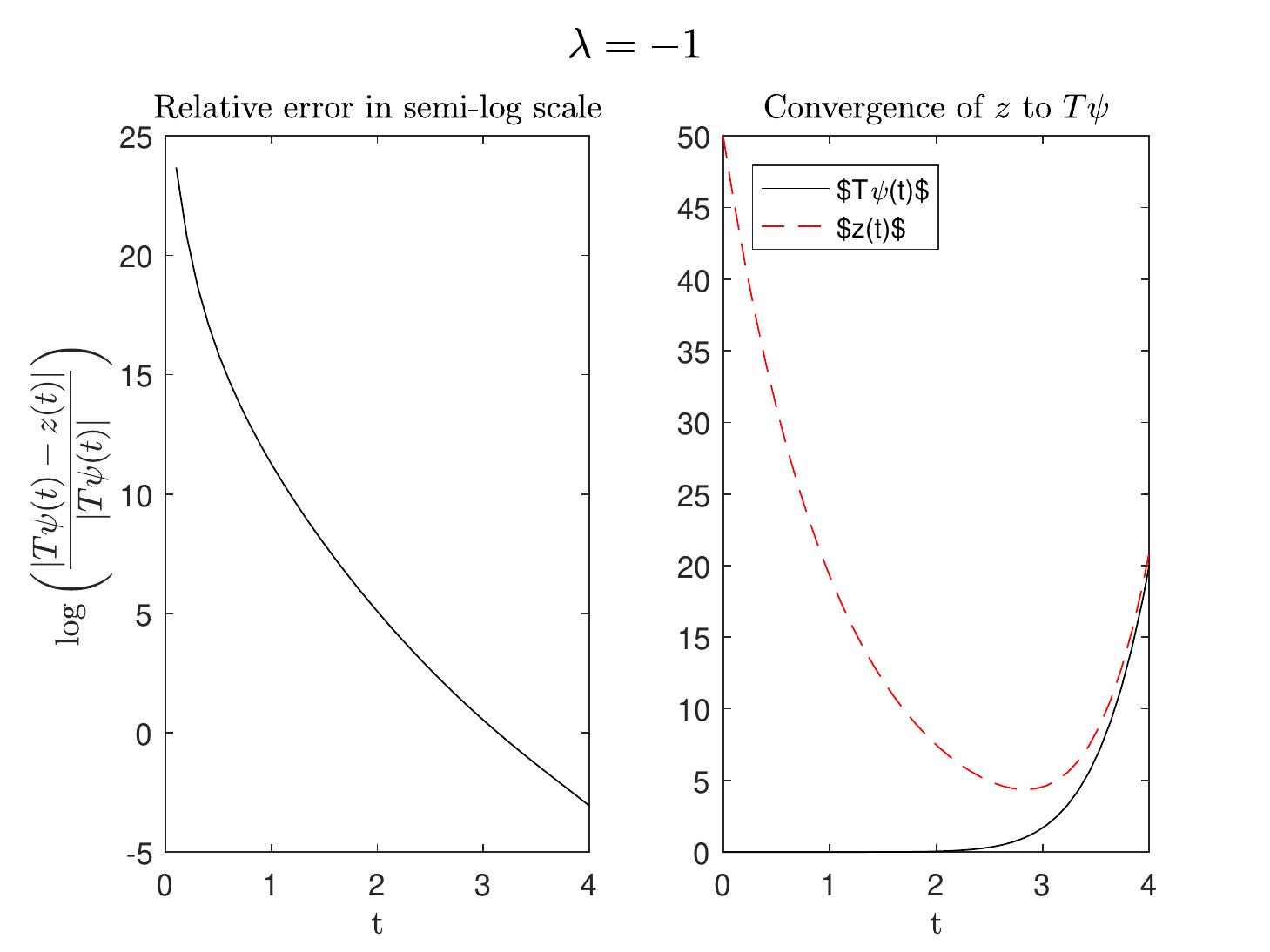}\\
\includegraphics[width = 0.5\textwidth]{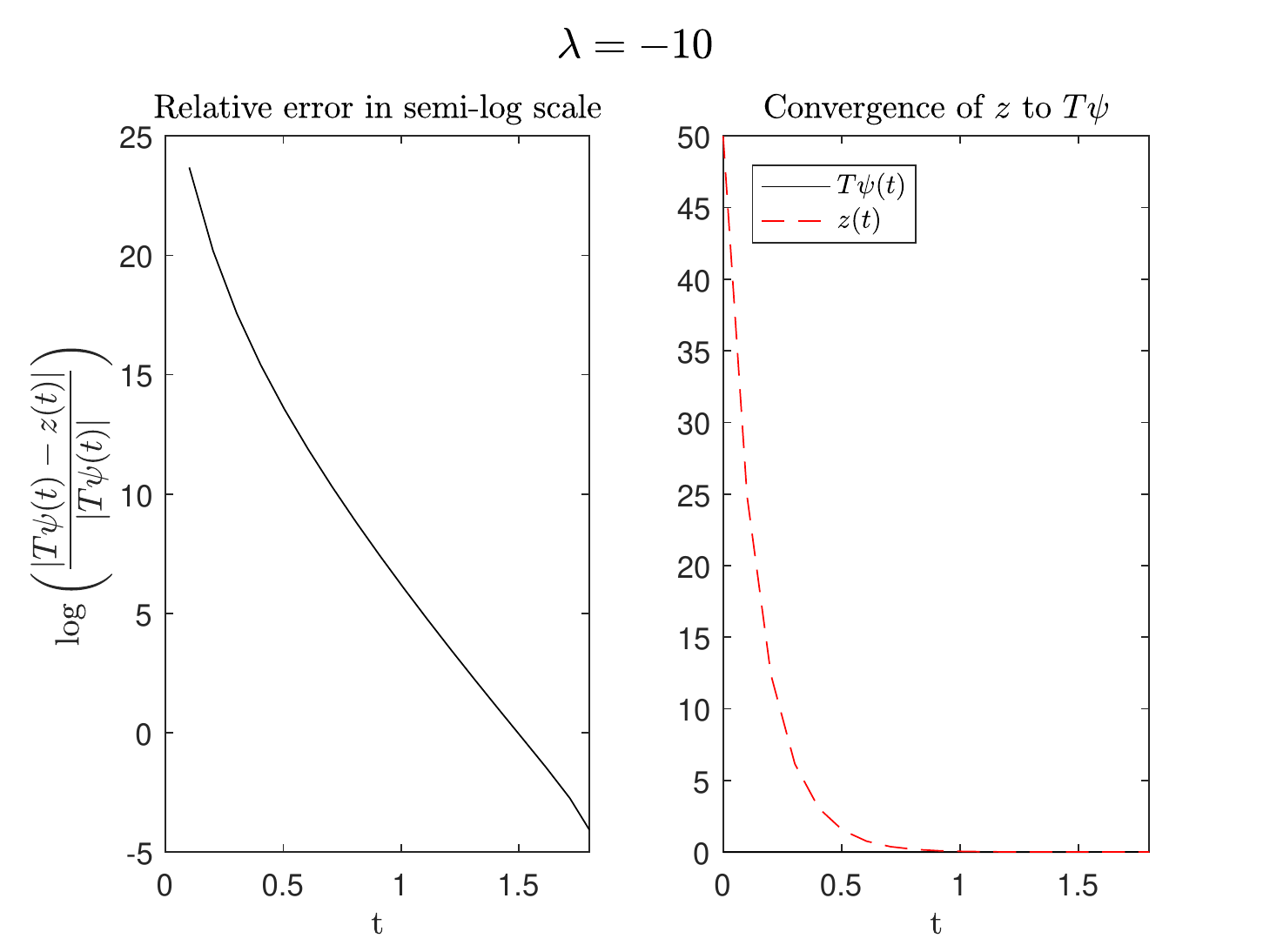}&
\includegraphics[width = 0.5\textwidth]{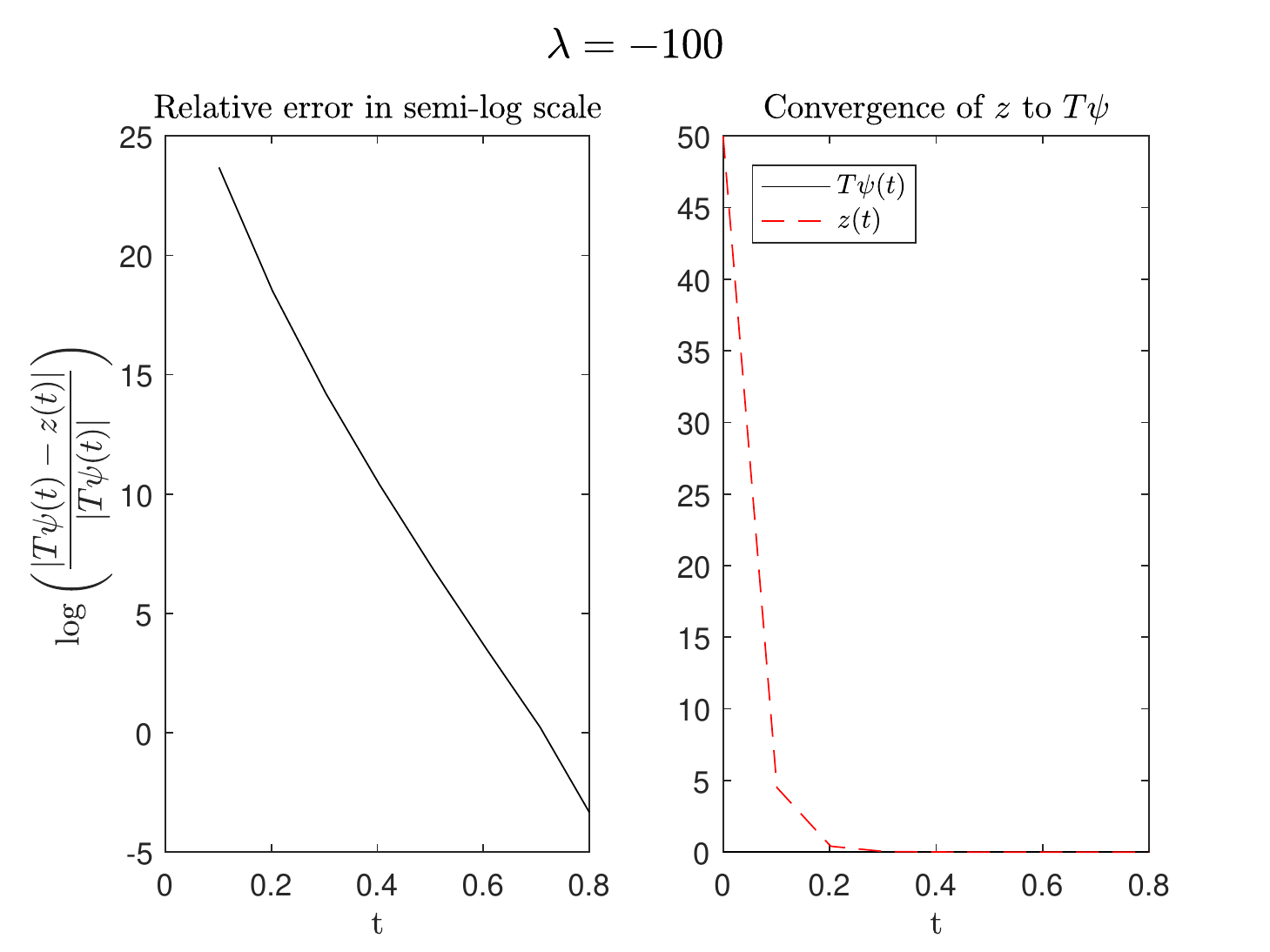}
\end{tabular}
\caption{\small Convergence of $\TR_\lambda(\etat) - z$ to zero for different values of $\lambda$.
We choose $z(0) \neq 0$ arbitrarily. The bigger is $|\lambda|$, the faster is the convergence. By means of a linear regression, one can estimate the convergence rate of the relative error to zero:
$
\mathcal{O}(e^{-7.4t})\text{ if } \lambda = -0,1,\quad
\mathcal{O}(e^{-8.2t})\text{ if } \lambda = -1,\quad
\mathcal{O}(e^{-14.2t})\text{ if } \lambda = -10,\quad
\mathcal{O}(e^{-32.9t})\text{ if } \lambda = -100.
$
}
\label{fig:func_state}
\end{figure}



\subsection{Step 2: reconstruction of the entire state of the system}

Following Step 1, we estimate simultaneously numerous functions $\TR_{\lambda_i}(\etat)$ which correspond to different values $\lambda_i <0$.
\startlucas
These estimations are denoted $z_{\lambda_i}$.
The aim of this section is to estimate the state $\psi$ from the knowledge of $(z_{\lambda_i})_{1\leq i\leq p}$.

Then, we choose a regularization parameter $\delta>0$ and solve the discrete version of the quadratic minimization problem \eqref{pb:min_well} at each time step, that is for each time $\mathbb{t_k}$, find $\big(\psi_{j, k}\big)_{1\leq j\leq N_x}$ that
\begin{equation}
    \textrm{minimize~} \left\|\Delta x \big(a_{i, j, k}\big)_{1\leq i\leq p, 1\leq j\leq N_x}\ \cdot\ \big(\psi_{j, k}\big)_{1\leq j\leq N_x} - \big(z_{\lambda_i}(\mathbb{t}_k)\big)_{1\leq i\leq p}\right\|^2+ \delta \left\| \big(\psi_{j, k}\big)_{1\leq j\leq N_x}\right\|^2.
\end{equation}

This is a quadratic minimization problem, which we solve via an interior-point method (see \emph{e.g.} \cite[Chapter III.11]{BV}.
We need to fix an initial condition to apply this algorithm.
Following a continuation method, we choose as an initial condition at time $\mathbb{t}_k$ the minimum value obtained at time $\mathbb{t}_{k-1}$, transported during a time $\Delta t$ at speed $G(\mathbb{t}_{k-1})$.

The choice of parameters $p$, $\lambda_1,\dots,\lambda_p$ and $\delta$ and their influence are investigated in the paragraphs below.
\begin{itemize}
    
    \item \textbf{Choice of the $p$ and $(\lambda_i)$.}
Note that the matrix $\big(a_{i, j, k}\big)_{1\leq i\leq p, 1\leq j\leq N_x}$ may be injective only if $p\geq N_x$, that is if the discretization in $\lambda$ is thinner than in $x$. Therefore, we fix $p= 2N_x = 200$.
Moreover, even if the matrix $(a_{i, j, k})_{i, j}$ is injective, a regularization method is needed to left-inverse it. Indeed, for all $t\in(t_0, t_1)$, the operator
$$L^2(\xmin, \xmax)\ni \psi\mapsto \left(\lambda\mapsto \int_{\xmin}^{\xmax} a_{\lambda}(t, x) \psi(x)\dd x \right)\in L^2(\lambda_{\min}, \lambda_{\max})$$
is compact (as an integral operator).
Hence, even if it is injective, its inverse is not continuous.
The matrix $(a_{i, j, k})_{i, j}$ is a discretization of this operator. Then, the more the discretization is thinner, the more it is ill-conditioned. This emphasizes the necessity of using a regularization method.

In Fig.~\ref{fig:lambda}, we plot the estimation of the NDF for different values of $(\lambda_i)$.
For large values of $|\lambda|$, $z$ converges quickly to $\TR\psi$. However, it appears that functions $a_\lambda$ carry less information for large values of $|\lambda|$, so that the map $\TR$ is more difficult to inverse. This explains Fig.~\ref{fig:llarge}, on which the estimation $\hat\psi$ is worst than on Fig.~\ref{fig:ltot}.

On the contrary, for small values of $|\lambda|$, it seems that functions $a_\lambda$ carry more information, since the estimation $\hat\psi$ is similar on Fig.~\ref{fig:lsmall} and Fig.~\ref{fig:ltot} at $t=10$ s. However, we also see a peaking phenomenon (for $t\leq5$ s on Fig.~\ref{fig:lsmall}), due to the fact that $z$ is slower to converge to $\TR\psi$ than for large values of $|\lambda|$.

Thus, one must find a compromise for the choice of $(\lambda_i)$: take large values for fast convergence and avoiding peaking, and small values for efficient estimation.

\begin{figure}[htbp]
\begin{subfigure}{.5\linewidth}
\includegraphics[width = \linewidth]{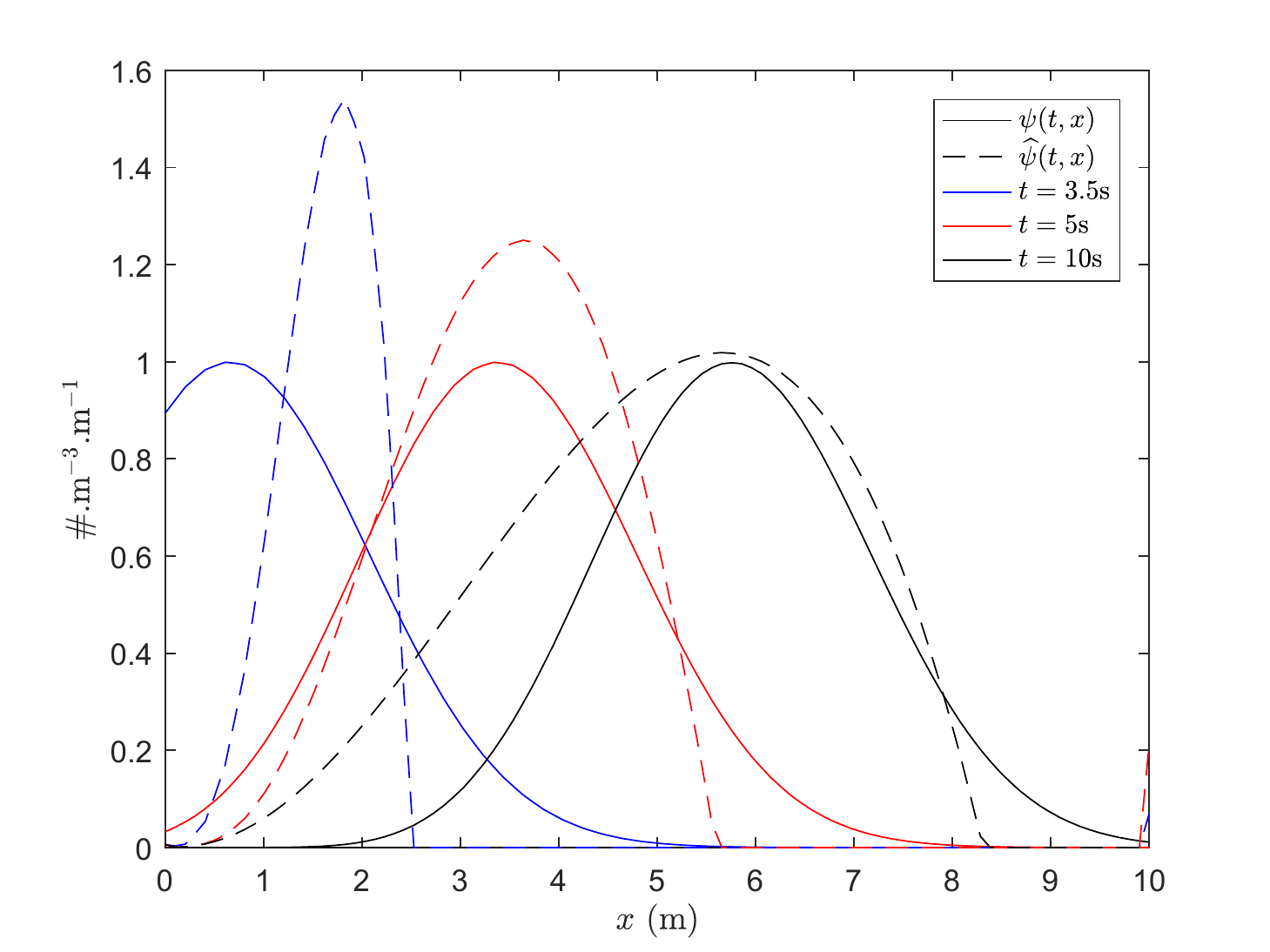}
\caption{$(\lambda_i)\subset[-10, -1]$}
\label{fig:lsmall}
\end{subfigure}
\begin{subfigure}{.5\linewidth}
\includegraphics[width = \linewidth]{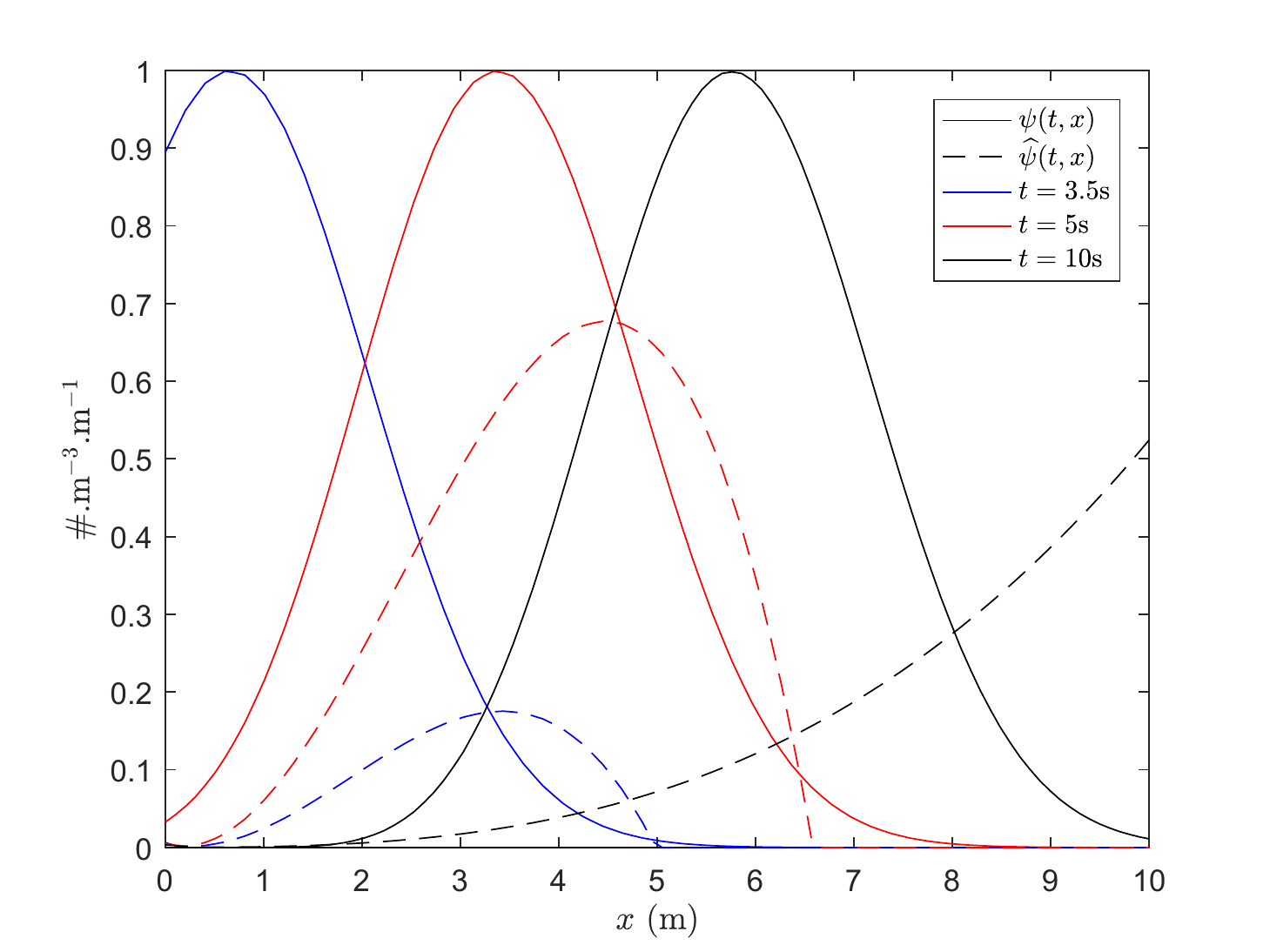}
\caption{$(\lambda_i)\subset[-100, -10]$}
\label{fig:llarge}
\end{subfigure}\par\medskip
\centering
\begin{subfigure}{.5\linewidth}
\includegraphics[width = \linewidth]{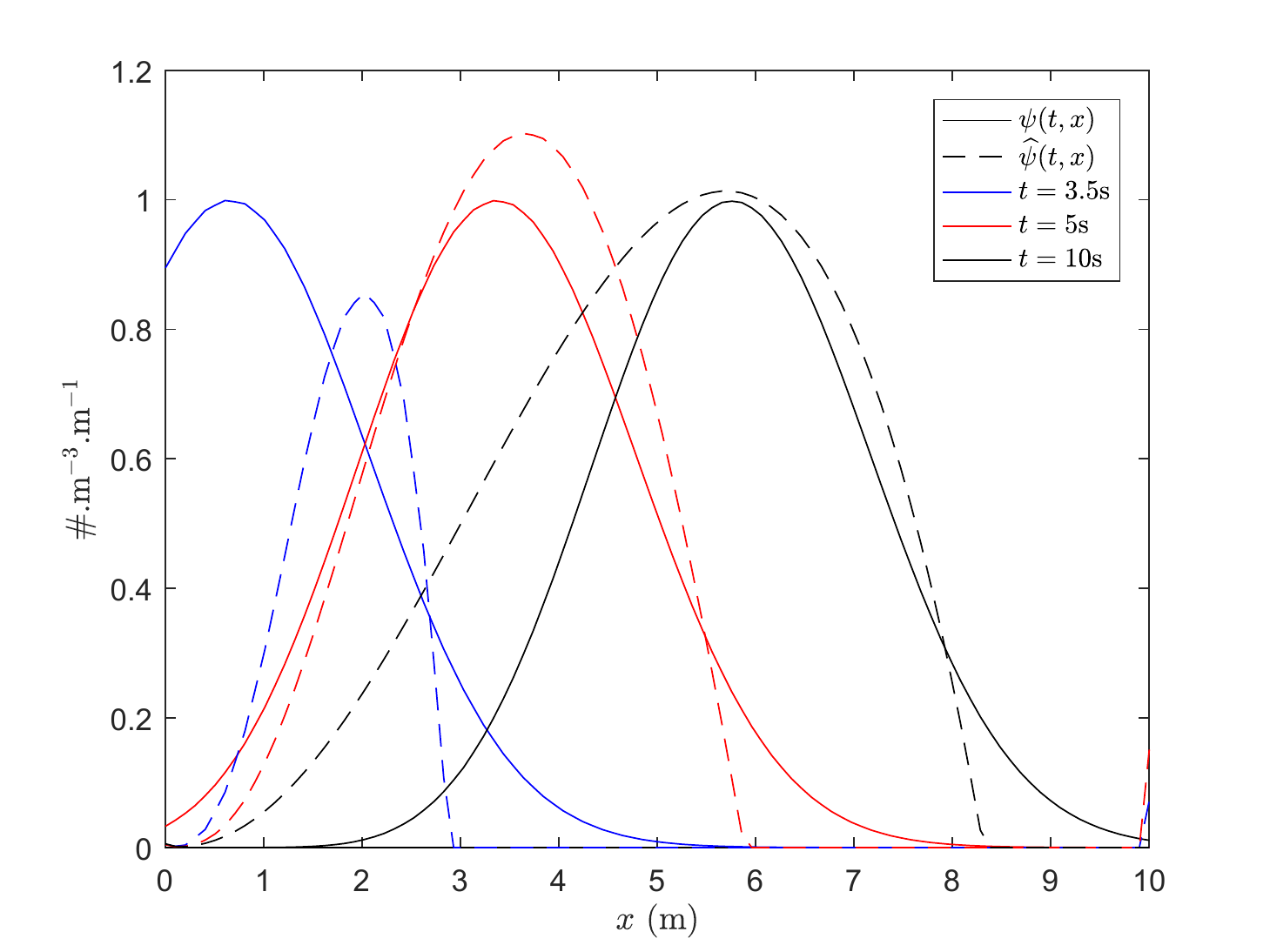}
\caption{$(\lambda_i)\subset[-100, -1]$}
\label{fig:ltot}
\end{subfigure}
\caption{Influence of $(\lambda_i)_{1\leq i\leq p}$ on the reconstruction of the CSD.}
\label{fig:lambda}
\end{figure}

    \item \textbf{Choice of the regularization parameter $\delta$.}
    
\end{itemize}

The regularization parameter $\delta$ must be chosen numerically, in order to find a compromise between the minimization of the norm of the state, and the minimization of the gap $\TR\psi - z$.
This compromise can be interpreted as a measurement reliability. Indeed, if the measurement has a small uncertainty, then we choose a small $\delta$. On the contrary, if the measurement is highly uncertain, then we fix a large value of $\delta$ in order to regularize the solution. 

In Fig.~\ref{fig:obs}, we plot the actual NDF $\etat$ and its estimation $\hat\etat$ at different times, for different values of $\delta$, and with or without measurement noise.
Measurement noise is fixed at $2\%$ of the maximal value of the output on the time interval.
For small values of $\delta$ and/or with measurement noise, we see that a peaking phenomenon appear: this is due to a lack of regularization of the solution.
On the contrary, if $\delta$ is too large, then the minimization of the norm of the state takes too much importance in the minimization problem, and $\hat\psi$ is too attenuated.

\stoplucas




\begin{figure}[htbp]
\centering
\begin{tabular}{r@{}l}
\subcaptionbox{$\delta=0.05,\ \alpha = 0$}{
\includegraphics[width = 0.5\textwidth]{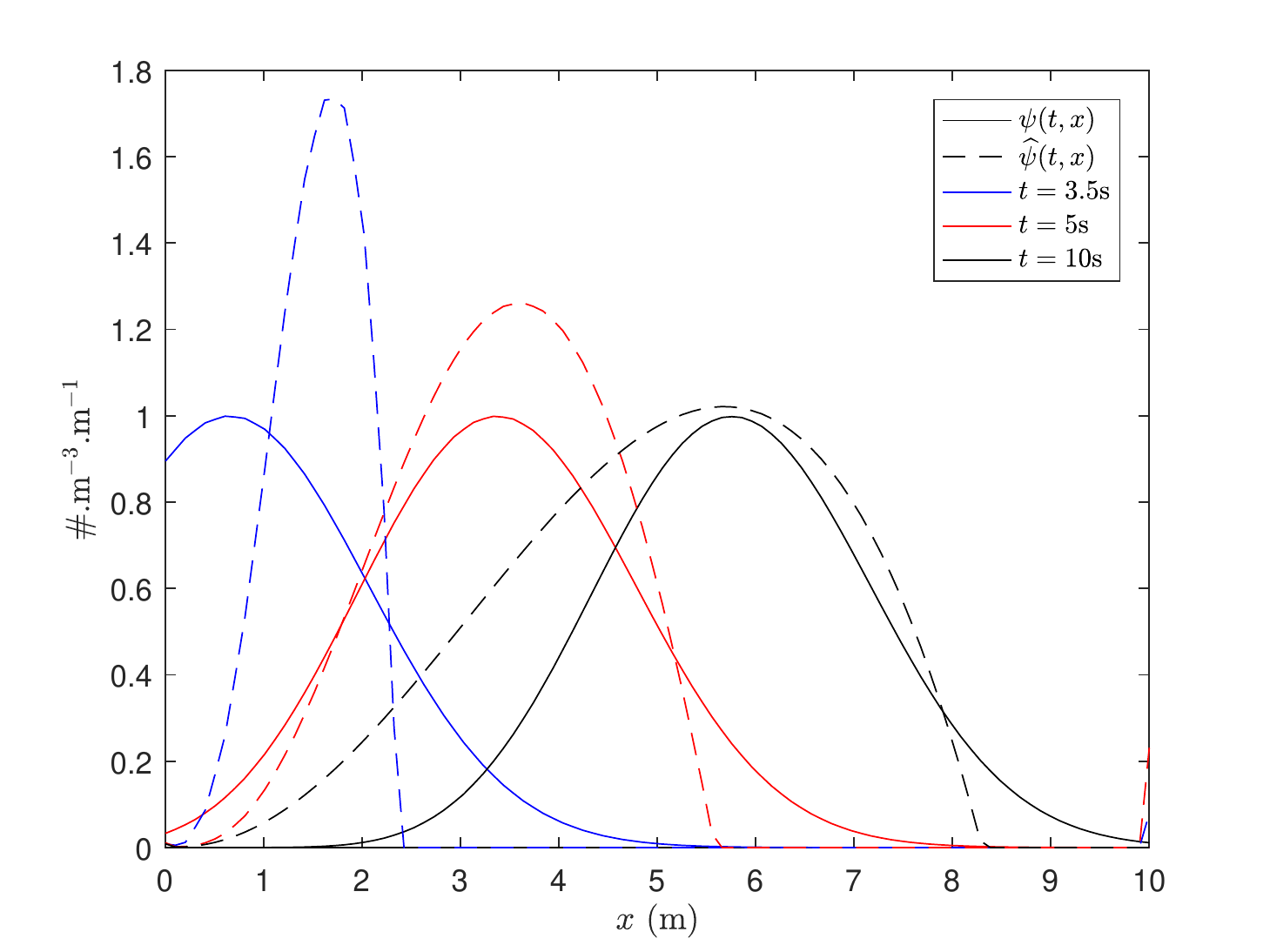}}&
\subcaptionbox{$\delta=0.05,\ \alpha = 2\%$}{
\includegraphics[width = 0.5\textwidth]{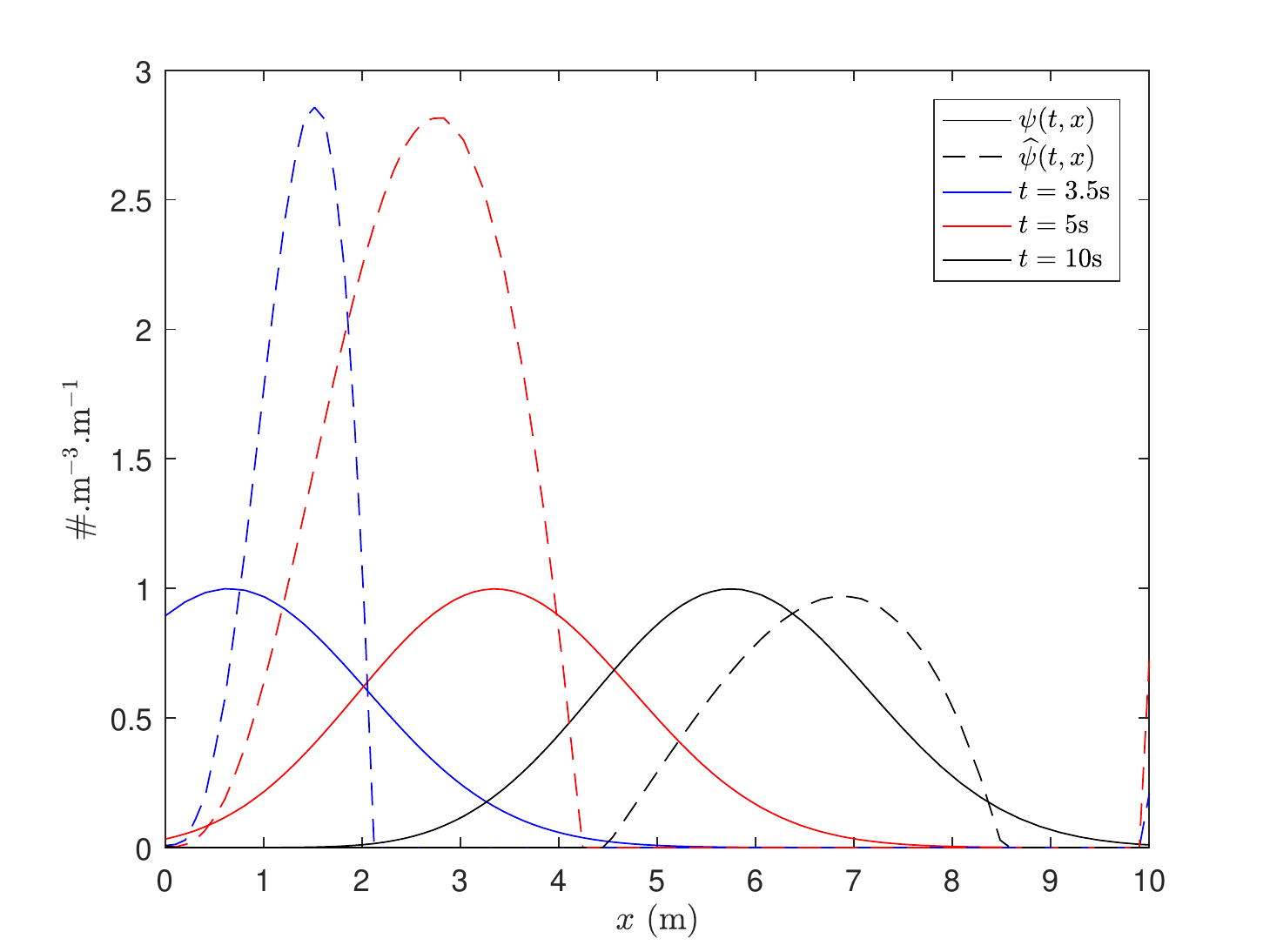}}\\
\subcaptionbox{$\delta=0.1,\ \alpha = 0$}{
\includegraphics[width = 0.5\textwidth]{delta01lambda1100bruit0.pdf}}&
\subcaptionbox{$\delta=0.1,\ \alpha = 2\%$}{
\includegraphics[width = 0.5\textwidth]{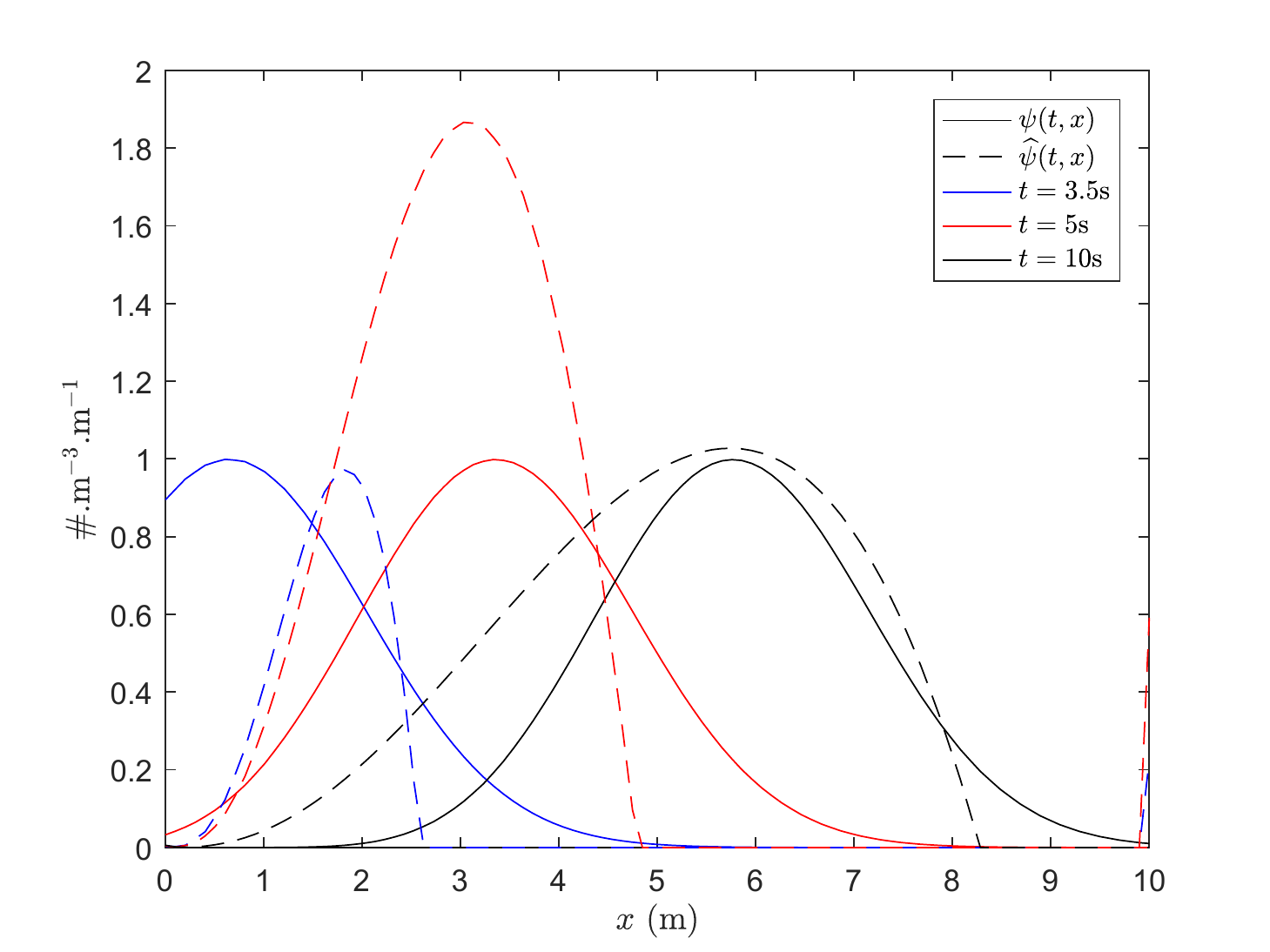}}\\
\subcaptionbox{$\delta=0.2,\ \alpha = 0$}{
\includegraphics[width = 0.5\textwidth]{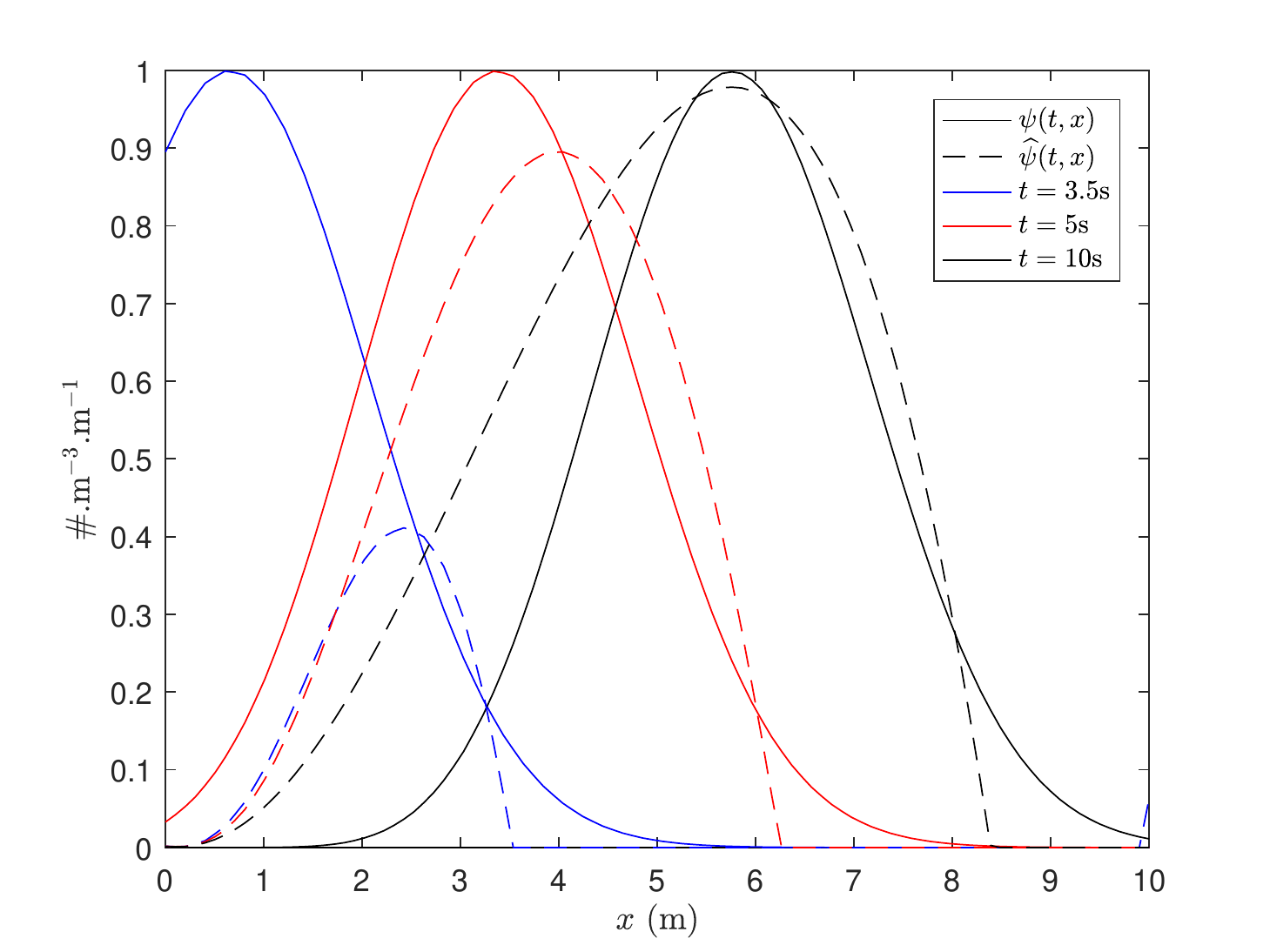}}&
\subcaptionbox{$\delta=0.2,\ \alpha = 2\%$}{
\includegraphics[width = 0.5\textwidth]{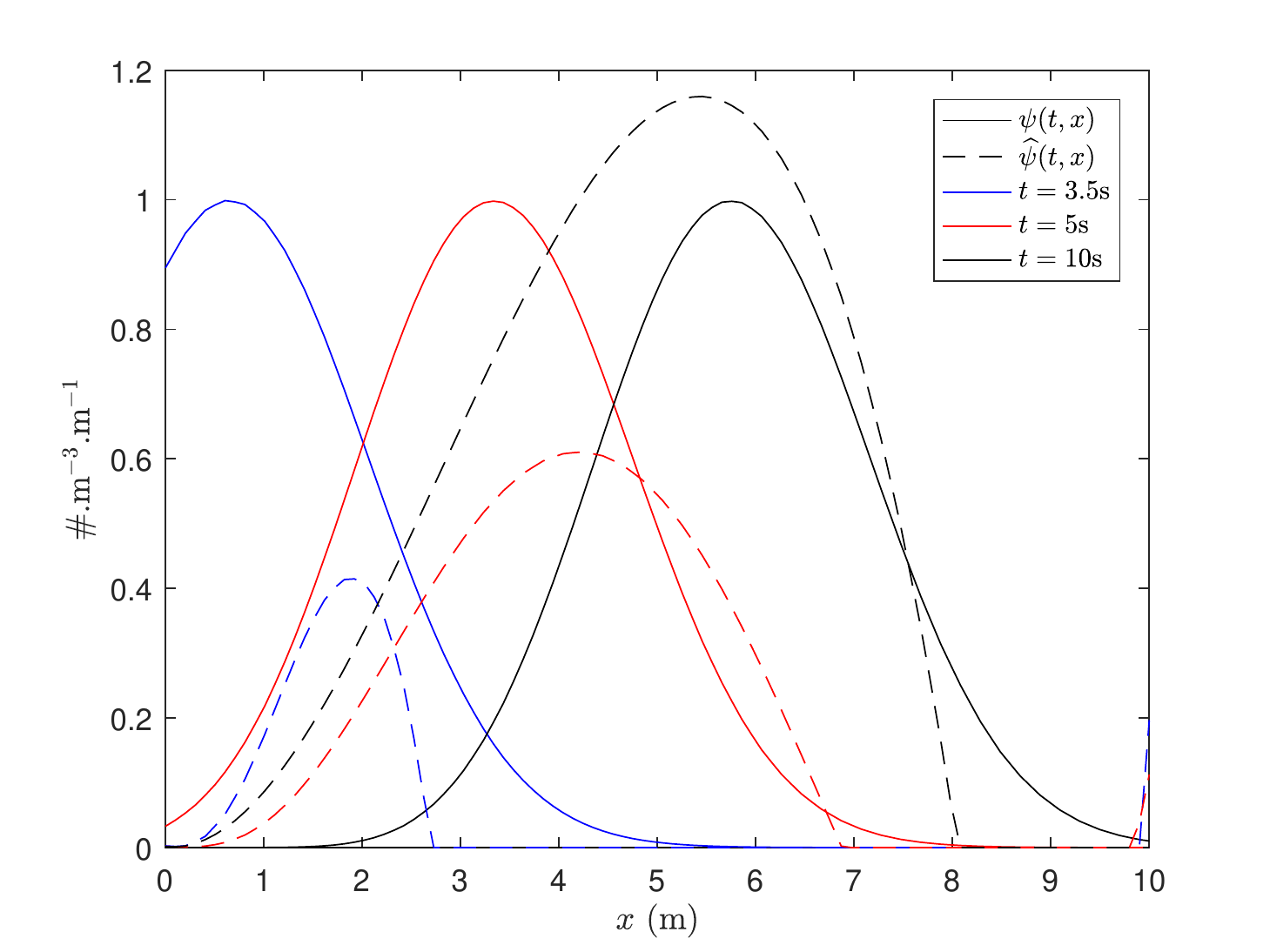}}
\end{tabular}
\caption{Influence of the regularization parameter $\delta$ and measurement noise $\alpha$ on the reconstruction of the CSD}
\label{fig:obs}
\end{figure}

\section{Conclusion}

In this paper a new observer has been introduced to estimate the Crystal Size Distribution
\startlucas
from the measurements of the solute concentration, temperature and a model of the growth rate.
No model of the nucleation rate is needed.
\stoplucas
This approach is based on the use of Kazantzis-Kravaris/Luenberger observer and a Tikhonov regularization procedure.
The numerical results obtained are promising. 
Even though, the knowledge of the solute concentration alone does not allow to accurately reconstruct the full CSD as shown by our observability analysis, we believe that this tool could be used in addition with supplementary information given by other sensors (for instance \fbrmp).
This approach will be evaluated on experiments in a future research project.

\section*{Appendix: proofs of Propositions \ref{prop_pos} and \ref{prop_neg}.}
\startlucas
Assume that $\vit\geq\mu>0$.
Let $\tau\in(t_0, t_1]$ and $u\in H^4(t_0, \tau)$.
Let $\etat$ be the solution of \eqref{systbilan} with initial condition $\etat_0$ and boundary condition $\cont$.
We introduce a time reparametrization $\Tilde{t} = \int_{t_0}^tG(s)\dd s$, which is well defined since $\vit\geq\mu$.
Let $\Tilde{\psi}$, $\Tilde{\cont}$ and $\Tilde{\mes}$ be such that
$\Tilde{\psi}(\Tilde{t}) = \psi(t)$, $\Tilde{\cont}(\Tilde{t}) = \cont(t)$ and $\Tilde{\mes}(\Tilde{t}) = \mes(t)$ for all $t\in[t_0, t_1]$. Then
\begin{equation}
\begin{aligned}
\begin{cases}
\partial_{\Tilde{t}} \Tilde{\etat}(\Tilde{t}, x) = -\partial_x \Tilde{\etat}(\Tilde{t}, x)\\
\Tilde{\etat}(t_0, x) = \etat_0(x)\\
\Tilde{\etat}(\Tilde{t}, \xmin) = \Tilde{\cont}(\Tilde{t})
\end{cases}
\end{aligned}
\label{syst_vit1}
\end{equation}
and $\Tilde{\mes}(\Tilde{t}) = \int_{\xmin}^{\xmax}\Tilde{\psi}(\Tilde{t}, x)x^3\dd x$.
Since the observability properties are not affected by the time reparametrization, one can investigate observability properties of the system \eqref{syst_vit1} instead of \eqref{systbilan}.
Therefore, one can assume without loss of generality that $\vit=1$ in the rest of the proof.
Since $\cont\in H^4(t_0, \tau)$, we have $\mes\in C^4(t_0, \tau)$.
Equation \eqref{eq_EtatMax} and system \eqref{systbilan} yield
\begin{align}
\mes'
&= 3 \int_{x_{\min}}^{x_{\max}}x^2\etat(\cdot, x)\dd x -  \left[x^3\etat(\cdot, x)\right]_{x_{\min}}^{x_{\max}}
\nonumber\\
&= 3 \int_{x_{\min}}^{x_{\max}}x^2\etat(\cdot, x)\dd x +  x_{\min}^3\cont,
\label{E:y(1)}
\\
\mes^{(2)}
&= 6 \int_{x_{\min}}^{x_{\max}}x\etat(\cdot, x)\dd x + 3 \left[x^2\etat(\cdot, x)\right]_{x_{\min}}^{x_{\max}}
\startlucas
+
\stoplucas
 x_{\min}^3\cont' \nonumber\\
&= 6 \int_{x_{\min}}^{x_{\max}}x\etat(\cdot, x)\dd x + 3 x_{\min}^2\cont +  x_{\min}^3\cont',
\label{E:y(2)}
\\
\mes^{(3)}
&= 6\int_{x_{\min}}^{x_{\max}}\etat(\cdot, x)\dd x - 6 \left[x\etat(\cdot, x)\right]_{x_{\min}}^{x_{\max}} + 3\ x_{\min}^2\cont' + x_{\min}^3\cont^{(2)} \nonumber\\
&= 6\int_{x_{\min}}^{x_{\max}}\etat(\cdot, x)\dd x + 6x_{\min}\cont + 3 x_{\min}^2\cont' +  x_{\min}^3\cont^{(2)},
\label{E:y(3)}
\\
\mes^{(4)}
&= - 6 \left[\etat(\cdot, x)\right]_{x_{\min}}^{x_{\max}} + 6x_{\min}\cont' + 3 x_{\min}^2\cont^{(2)}(t) + x_{\min}^3\cont^{(3)} \nonumber\\
&= 6\cont + 6x_{\min}\cont' + 3 x_{\min}^2\cont^{(2)} +  x_{\min}^3\cont^{(3)}.
\label{E:y(4)}
\end{align}

\medskip
\begin{proof}[End of the proof of Proposition \ref{prop_pos}]
By hypothesis, $\etat_0=0$.
Consequently, Equations \eqref{E:y(1)}--\eqref{E:y(3)} yield
\begin{align*}
\begin{cases}
\mes'(t_0) = x_{\min}^3\cont(t_0)\\
\mes^{(2)}(t_0) = 3 x_{\min}^2\cont(t_0) + x_{\min}^3\cont'(t_0)\\
\mes^{(3)}(t_0) = 6 x_{\min}\cont(t_0) + 3 x_{\min}^2\cont'(t_0) +x_{\min}^3\cont^{(2)}(t_0),
\end{cases}
\end{align*}
which is a triangular system with non vanishing diagonal since $\xmin>0$. Hence $\cont(t_0)$, $\cont'(t_0)$ and $\cont^{(2)}(t_0)$ are determined by $\mes$.
Moreover, on $[t_0, \tau]$, $\cont$ satisfies Equation \eqref{E:y(4)} which is a 3rd order ordinary differential equation.
Hence, according to the Cauchy-Lipschitz theorem, there exits a unique solution $\cont$ to this problem.
Thus $\mes$ determines $\cont$ uniquely, that $\cont\mapsto\mes$ is injective.
\end{proof}

\begin{proof}[End of the proof of Proposition \ref{prop_neg}]
Substituting the boundary condition in equation \eqref{E:y(4)} with  $u=0$ yields $\mes^{(4)} = 0$ identically on $[t_0, \tau]$.
Hence $\mes$ is a polynomial function of degree less or equal than 3.
Thus the linear function that maps any solution of \eqref{systbilan} with null boundary condition to its third moment has rank 4. Since $\etat$ lies in an infinite dimensional vector space, we get by the rank-nullity theorem that its kernel is non-trivial, \emph{i.e.} the state-output map $\etat\mapsto\mes$ is not injective, and the system has a 4-dimensional observable part.
\end{proof}

Note that Proposition \ref{prop_neg} relies deeply on Hypothesis \eqref{eq_EtatMax}.
Hence the non-injectivity of the measurement is due to the fact that the system is observed on a too small time interval.
If the system was observed on $[t_0, +\infty)$, then one could show with similar argues an injectivity result.
\stoplucas

\bibliographystyle{abbrv}
\bibliography{Biblio}

\end{document}